\documentclass[final,11pt]{elsarticle}

\usepackage{verbatim,a4wide}

\usepackage{amsmath}
\usepackage{amssymb}
\usepackage{amsthm}

\usepackage{algorithm}
\usepackage{algpseudocode}

\usepackage{xcolor}

\usepackage{tikz}
\usetikzlibrary{patterns}

\usepackage{hyperref}

\usepackage[cp1250]{inputenc}
\usepackage[OT4]{fontenc}
\usepackage[english]{babel}

\numberwithin{equation}{section}
\numberwithin{algorithm}{section}
\numberwithin{figure}{section}
\numberwithin{table}{section}

\newtheorem{theorem}{Theorem}[section]

\newtheorem{remark}[theorem]{Remark}

\newtheorem{example}[theorem]{Example}
\newtheorem{definition}[theorem]{Definition}
\newtheorem{problem}[theorem]{Problem}

\newcommand{\p}[1]{\mbox{\textsf{#1}}}

\begin{document}

\begin{frontmatter}

\title{Linear-time algorithm for computing the
Bernstein-B\'{e}zier coefficients of B-spline basis functions}

\author{Filip Chudy\corref{cor}}
\ead{Filip.Chudy@cs.uni.wroc.pl}
\cortext[cor]{Corresponding author.}

\author{Pawe{\l} Wo\'{z}ny}
\ead{Pawel.Wozny@cs.uni.wroc.pl}

\address{Institute of Computer Science, University of Wroc{\l}aw,
         ul.~Joliot-Curie 15, 50-383 Wroc{\l}aw, Poland}

\begin{abstract}
A new differential-recurrence relation for the B-spline functions of the same
degree is proved. From this relation, a recursive method of computing the
coefficients of B-spline functions of degree $m$ in the Bernstein-B\'{e}zier 
form is derived. Its complexity is proportional to the number of coefficients in
the case of coincident boundary knots. This means that, asymptotically, the
algorithm is optimal. In other cases, the complexity is increased by at most
$O(m^3)$. When the Bernstein-B\'{e}zier coefficients of B-spline basis functions
are known, it is possible to compute any B-spline function in linear time with
respect to its degree by performing the geometric algorithm proposed recently by
the authors.
This algorithm scales well when evaluating the
B-spline curve at multiple points, e.g., in order to render it,
since one only needs to
find the coefficients for each knot span once. When evaluating
many B-spline curves at multiple points (as is the case when rendering
tensor product B-spline surfaces), such approach has lower computational
complexity than using the de Boor-Cox algorithm.
The numerical tests show that the new method is efficient.
The problem of finding the
coefficients of the B-spline functions in the power basis can be solved
similarly.
\end{abstract}

%
%
%
%

\begin{keyword}
B-spline functions, Bernstein-B\'{e}zier form, recurrence relations,
B\'{e}zier curves, B-spline curves, B-spline surfaces, de Boor-Cox algorithm.
\end{keyword}

\end{frontmatter}

\section{Introduction}                                  \label{S:Introduction}

The family of Bernstein (basis) polynomials was used by S.~N.~Bernstein in 1912
in his constructive proof of the Weierstrass approximation theorem.
For more details, see~\cite[\S10.3]{DD02}.
Half a century later, they came into prominence when they became
a basis for a particular family of parametric curves
--- \textit{B\'{e}zier curves}.

For a fixed $n \in \mathbb{N}$ and $0\leq i\leq n$, $B^n_i$ is the
\textit{$i$th Bernstein (basis) polynomial of degree $n$} given by the formula
\begin{equation}\label{E:Def_BernPoly}
 B^n_i(t) := \binom{n}{i} t^i (1-t)^{n-i}.
\end{equation}

\begin{remark}
In the sequel, the convention is applied that $B^n_i\equiv 0$ if $i<0$ or $i>n$. 
\end{remark}

The Bernstein polynomial $B^n_i$ $(0\leq i\leq n)$ is non-negative and has
exactly one maximum value in the interval $[0,1]$ (except for the case $n=0$). It
is also important that all Bernstein polynomials of the same degree give the
partition of unity, i.e., $\sum_{i=0}^{n}B^n_i(t)\equiv 1$. These properties
make the Bernstein basis a very powerful tool in approximation theory,
numerical analysis or in CAGD.

For any $t \in \mathbb{R}$, Bernstein polynomials satisfy the recurrence
relations connecting the polynomials of two consecutive degrees:
\begin{eqnarray}
\label{E:BernDegDown}
B^n_k(t) & = & t B^{n-1}_{k-1}(t)
          + (1-t) B^{n-1}_k(t),\\
\label{E:BernDegUp}
B^n_k(t) & = & \dfrac{n-k+1}{n+1} B^{n+1}_{k}(t)
          + \dfrac{k+1}{n+1} B^{n+1}_{k+1}(t) \qquad (0\leq k\leq n).
\end{eqnarray}
It is also well-known that
\begin{equation}\label{E:BernsteinDerivSimple}
 \Big(B^n_k(t)\Big)' = n \Big(B^{n-1}_{k-1}(t) - B^{n-1}_k(t) \Big) 
                                                        \qquad (0\leq k\leq n).
\end{equation}
See~\cite[\S5.1, Properties 2, 11, 13]{Farouki2012}.

Using Eqs.~\eqref{E:Def_BernPoly}, \eqref{E:BernDegUp},
and \eqref{E:BernsteinDerivSimple}, one can easily obtain the following identities:
\begin{eqnarray} 
\label{E:BernsteinuDeriv}
 t \Big(B^n_k(t)\Big)' & = & k B^n_k(t) - (k+1) B^n_{k+1}(t),\\
\label{E:BernsteinDeriv}
 \Big(B^n_k(t)\Big)' & = & (n-k+1) B^n_{k-1}(t) + (2k - n) B^n_k(t)
                     - (k+1) B^n_{k+1}(t),
\end{eqnarray}
where $k=0,1,\ldots,n$.

\begin{remark}
 Points in $\mathbb{E}^d$, as well as parametric objects
 which return a point in $\mathbb{E}^d$, are denoted using
 the upper-case letters in the sans-serif font, e.g.,
 $\p{S}, \p{P}, \ldots$ (as opposed to, e.g., Bernstein polynomials $B^n_i$,
 which return real values).
\end{remark}

Polynomial B\'{e}zier curves are a particular family of parametric curves which
is defined as a convex combination of control points. The points are weighted
using Bernstein polynomials~\eqref{E:Def_BernPoly}. A \textit{B\'{e}zier curve}
$\p{P}_n: [0, 1] \rightarrow \mathbb{E}^d$ \textit{of degree $n$} with
\textit{control points} $\p{W}_0, \p{W}_1, \ldots, \p{W}_n \in \mathbb{E}^d$ is
defined by the formula
\begin{equation}\label{E:BezierCurve}
  \p{P}_n(t) := \sum_{k=0}^n B^n_k(t) \p{W}_k\qquad (0\leq t\leq 1).
\end{equation}
Due to the properties of Bernstein polynomials,
the curve $\p{P}_n$ is in the \textit{convex hull} of the control
points which means that 
$\p{P}_n([0,1])\subseteq\mbox{conv}\{\p{W}_0,\p{W}_1,\ldots,\p{W}_n\}$.

One can evaluate a point on a B\'{e}zier curve using the famous de Casteljau
algorithm which is based on the relation~\eqref{E:BernDegDown} and has $O(dn^2)$
complexity. 

A new method for computing a point on a polynomial or rational B\'{e}zier curve
in optimal $O(dn)$ time has been recently proposed by the authors
in~\cite{WCh2020}. The new algorithm combines the qualities of previously known
methods for solving this problem, i.e., the linear complexity of the Horner's
scheme and the geometric interpretation, the convex hull property, and operating
only on convex combinations which are the advantages of the de Casteljau
algorithm. The new method can be used not only for polynomial and
rational B\'{e}zier curves but also for other \textit{rational parametric
objects}, e.g., B\'{e}zier surfaces.
For more details, see~\cite[Chapter~2]{FChPhD}.

For more information about the history
of B\'{e}zier curves and their properties, see, e.g.,
\cite{Bezier66, Bezier67, Bezier68, BM99, FChPhD, DC63, DC59, DC99}, as well as
\cite[\S1]{Farin2002} and \cite[\S4]{Farouki2012}.

Despite their elegance and some desirable properties, Bernstein polynomials
have a significant drawback. For any $n, i \in \mathbb{N}$ such that 
$0\leq i \leq n$, the value of a Bernstein polynomial $B^n_i(t)$ is non-zero for
all $t\in(0, 1)$. In practice, when operating on a B\'{e}zier 
curve~\eqref{E:BezierCurve}, any change in one control point's position changes
the curve over its whole length.

To address this issue, \textit{B-spline functions} can be used. They are
constructed in a way which eliminates this drawback. When used as a basis for
parametric curves (known as \textit{B-spline curves}), any change to
the control point only has a local effect on a curve.

In order to introduce a B-spline function, the \textit{generalized divided
differences} will be used.
\begin{definition}[{\cite[\S4.2.1]{DB}}]
The {\rm generalized divided difference} of a univariate function $f$ at the
knots $x_i, x_{i+1}, \ldots, x_k$ (which may be coincident), denoted by
$[x_i,x_{i+1},\ldots,x_k]f$, is defined in the following recursive way:
\begin{equation*}
[x_i,x_{i+1},\ldots,x_{i+\ell}]f := 
  \left\{
   \begin{array}{ll}
   \dfrac{[x_{i+1},\ldots,x_{i+\ell}]f-[x_i,\ldots,x_{i+\ell-1}]f}
          {x_{i+\ell}-x_i} & (x_i \neq x_{i+\ell}),\\
   \dfrac{f^{(\ell)}(x_i)}{\ell!} & (x_i = \ldots = x_{i+\ell}).
  \end{array}
  \right.
\end{equation*}
In particular, $[x_i]f = \dfrac{f^{(0)}(x_i)}{0!} = f(x_i)$.
\end{definition}

\begin{definition}[{\cite[\S5.11]{Prautzsch}}]\label{D:BSplineDivDiff}
The {\rm B-spline (basis) function $N_{mi}$ of degree $m\in\mathbb N$} with knots
$t_i \leq t_{i+1} \leq \ldots \leq t_{i+m+1}$ is defined as
$$
N_{mi}(u) := (t_{i+m+1}-t_i) [t_i, t_{i+1}, \ldots, t_{i+m+1}](t-u)^m_+,
$$
where the generalized divided difference acts on the variable $t$, and
$$
(x-c)^m_+ := \left\{\begin{array}{ll}
                    (x-c)^m & (x \geq c),\\
                    0       & (x < c)
                   \end{array}
\right.
$$
is the {\rm truncated power function}.
\end{definition}

The B-spline function $N_{mi}$ with knots $t_i \leq t_{i+1} \leq \cdots \leq
t_{m+i+1}$ has support $[t_i, t_{m+i+1}]$, i.e., $N_{mi}(u)$ can be non-zero 
only for $u \in [t_i, t_{m+i+1}]$ (see, e.g., \cite[Property 2.2]{NURBS}).

Let $m,n \in \mathbb{N}$. The knots
$$
 \underbrace{t_{-m} \leq \ldots \leq t_{-1} \leq t_0}_\text{boundary knots}
 \leq
 \underbrace{t_1 \leq \ldots \leq t_{n-1}}_\text{inner knots}
 \leq
 \underbrace{t_n \leq t_{n+1} \leq \ldots \leq t_{n+m}}_\text{boundary knots},
$$
where $t_0 < t_n$, serve as a support for
a B-spline basis of degree $m$ over $[t_0, t_n]$. The B-spline functions
$N_{m,-m}, N_{m,-m+1}, \ldots, N_{m,n-1}$ 
(cf.~Definition~\ref{D:BSplineDivDiff}) form a B-spline basis.
Splines are commonly used in a wide variety of applications, e.g., in
computer-aided geometric design, approximation theory and numerical analysis.
See, e.g., \cite{Dierckx93, Farin2002, Goldman, NURBS, Prautzsch}.

\begin{remark}\label{R:BSplineConvention}
In the sequel, a convention is adopted that for any quantity $Q$, if 
$t_k = t_{m+k+1}$ then $\dfrac{Q}{t_{m+k+1}-t_{k}}:=0$, as well as that
$N_{mi}\equiv0$ for $i<-m$ or $i\geq n$.
\end{remark}

Computing the B-spline functions or their derivatives using 
Definition~\ref{D:BSplineDivDiff} is costly. Instead, one
can use the recurrence and differential-recurrence relations.

The B-spline functions satisfy the following de Boor-Mansfield-Cox 
recursion formula (see, e.g., \cite[Eq.~(7.8)]{Goldman}, \cite[\S2]{DeBoor1972},
\cite[Eq.~(6.1)]{Cox1972}):
\begin{equation}\label{E:BSplineRec}
N_{mi}(u) = 
  (u - t_i) \dfrac{N_{m-1, i}(u)}{t_{m+i} - t_i} 
        + (t_{m+i+1} - u) \dfrac{N_{m-1, i+1}(u)}{t_{m+i+1} - t_{i+1}}
                                                           \qquad(-m\leq i<n)
\end{equation}
(cf.~Remark~\ref{R:BSplineConvention}).
Additionally, for $i=0,1,\ldots,n-1$,
\begin{equation}\label{E:BSplineN0iExplicit}
 N_{0i}(u) = 
\left\{
\begin{array}{ll}
 1 & (u \in [t_i, t_{i+1})),\\
 0 & \mbox{otherwise}.
\end{array}
\right.
\end{equation}
The derivative of a B-spline function can be expressed as
\begin{equation}\label{E:BSplineDeriv}
N_{mi}'(u) = 
  m \left(\dfrac{N_{m-1, i}(u)}{t_{m+i} - t_i}
        - \dfrac{N_{m-1, i+1}(u)}{t_{m+i+1} - t_{i+1}}\right)
                                                     \qquad(-m\leq i<n)
\end{equation}
(cf.~\cite[Eq.~(2.7)]{NURBS} and Remark~\ref{R:BSplineConvention}).

\begin{theorem}[{\cite[Property 2.5]{NURBS}}]\label{T:BSplineMultiDeriv}
All derivatives of $N_{mi}$ exist in the interior of a knot span (where it is 
a polynomial). At a knot $N_{mi}$ is $m-k$ times continuously differentiable,
where $k$ is the multiplicity of the knot. Hence, increasing $m$ increases
continuity, and increasing knot multiplicity decreases continuity.
\end{theorem}

The B-spline functions, like the family of Bernstein polynomials of an arbitrary
degree, have properties which make them a good choice for a parameterization of 
a family of curves.

\begin{theorem}[{\cite[Properties 2.3, 2.4, 2.6]{NURBS}}]
                                                     \label{T:BSplineNonNegSum1}
$N_{mi}(u) \geq 0$ for all $m, i, u$ (non-negativity). For an arbitrary knot
span, $[t_j, t_{j+1})$, $\sum_{i=j-m}^j N_{mi}(u) = 1$ for all 
$u \in [t_j, t_{j+1})$ (partition of unity). Except for the case $m = 0$,
$N_{mi}$ attains exactly one maximum value.
\end{theorem}

A \textit{B-spline curve of degree $m$} over the non-empty interval $[a, b]$
with knots 
$$
t_{-m} \leq \ldots \leq t_0 = a \leq t_1 \leq \ldots \leq b 
                                                = t_n \leq \ldots \leq t_{n+m}
$$ 
and \textit{control points}
$\p{W}_{-m}, \p{W}_{-m+1}, \ldots, \p{W}_{n-1}\in{\mathbb E^d}$ is defined as
$$
\p{S}(t) := \sum_{i=-m}^{n-1} N_{mi}(t) \p{W}_i \qquad (t \in [a, b]).
$$
One can check that
$\p{S}([a,b])\subseteq\mbox{conv}\{\p{W}_{-m},\p{W}_{-m+1},\ldots,\p{W}_{n-1}\}$,
as well as
$\p{S}([t_i,t_{i+1}])\subseteq
                      \mbox{conv}\{\p{W}_{i-m},\p{W}_{i-m+1},\ldots,\p{W}_i\}$ 
for $0\leq i\leq n-1$ (see, e.g., \cite[Property 3.5]{NURBS}).

The recurrence relation~\eqref{E:BSplineRec} and Eq.~\eqref{E:BSplineN0iExplicit}
can be used to evaluate a point on a B-spline curve. This approach, applied to
explicitly compute the values of B-spline functions, has been proposed by de
Boor in~\cite[p.~55--57]{DeBoor1972}. The algorithm given
in~\cite[p.~57--59]{DeBoor1972} (see also, e.g,~\cite[Eq.~(8.3)]{Farin2002}),
which directly computes a point on a B-spline curve is known as the de Boor-Cox
algorithm and has $O(dm^2)$ computational complexity.

A popular choice for the boundary knots is to make them \textit{coincident} with
$t_0$ and $t_n$, i.e.,
\begin{equation}\label{E:KnotClones}
t_{-m} = t_{-m+1} = \ldots = t_{-1} = t_0 = a,\quad 
                                          b=t_n = t_{n+1} = \ldots = t_{n+m}.
\end{equation}
In this case, $\p{S}(a) = \p{W}_{-m}$ and $\p{S}(b) = \p{W}_{n-1}$.

B\'{e}zier curves are a particular
subtype of B-spline curves for $n=1$,
$t_{-m} = t_{-m + 1} = \ldots = t_0 = 0$ and $t_1 = t_2 = \ldots = t_{m+1} = 1$,
$$
\p{S}(t)=\sum_{i=-m}^0N_{mi}(t)\p{W}_i=\sum_{i=0}^mB^m_i(t)\p{W}_{i-m}
                                                          \qquad (t \in [0, 1]).
$$
See, e.g., \cite[Property 3.1]{NURBS}.

The paper is organized as follows.
Section~\ref{S:BB-P-Coeff-BplineF} comprises the problem statement
for finding the adjusted Bernstein-B\'{e}zier coefficients
of B-spline functions over each knot span.
In Section~\ref{S:BSplineDiffRec} we prove
the new differential-recurrence relation between the B-spline functions of the
same degree. It will be the foundation for new recurrence relations which
can be used to formulate an algorithm which computes the required coefficients.
In Section~\ref{S:BSplineBezier}, the algorithm for finding
the adjusted Bernstein-B\'{e}zier coefficients of all B-spline functions
over each knot span
in the case
$t_{-m} = t_0$, $t_n = t_{n+m}$ and all inner knots
$t_1, t_2, \ldots, t_{n-1}$ having multiplicity
$1$ is given. The computational complexity of the method is $O(nm^2)$. This means
that, asymptotically, the algorithm is optimal. Section~\ref{S:Applications}
expands upon using the new algorithm to compute multiple points on multiple
B-spline curves (\S\ref{SS:BSplineFastComp}).
The numerical tests which are presented there show that the new method
is efficient.
In \S\ref{SS:SurfaceFastComp}, the results of applying the new algorithm
to evaluating tensor product B-spline surfaces are given.
The assumptions about knot multiplicity which were made in
Section~\ref{S:BSplineBezier} are then relaxed in Section~\ref{S:BSplineGen} to
cover all cases (cf.~Remark~\ref{R:BSplineAssumptions}).
Section~\ref{S:UniformKnots} gives a simplified version of
the recurrence relations for the case of uniform knots.

\section{The Problem: Bernstein-B\'{e}zier and power coefficients of B-spline
                                                                    functions}                                 
                                                  \label{S:BB-P-Coeff-BplineF}

In this paper, the following problem is considered.

Let the \textit{adjusted Bernstein-B\'{e}zier basis form} of the B-spline 
function $N_{mi}$ over a single non-empty \textit{knot span} 
$[t_j, t_{j+1}) \subset [t_0, t_n]$ ($j=i,i+1,\ldots,i+m$) be
\begin{equation}\label{E:BSplineABB}
N_{mi}(u) = \sum_{k=0}^m b^{(i,j)}_k B^m_k\Big(\dfrac{u - t_j}
                                                     {t_{j+1} - t_j}\Big)
                                                \qquad (t_j \leq u < t_{j+1}),                                                     
\end{equation}
with $b_k^{(i,j)} \equiv b_{k,m}^{(i,j)}$. 

\begin{problem}\label{P:BSpline1}
Find the adjusted Bernstein-B\'{e}zier basis coefficients $b_k^{(i,j)}$
$(0\leq k\leq m)$ (cf.~\eqref{E:BSplineABB}) of all functions $N_{mi}$ over all
non-trivial knot spans $[t_j, t_{j+1}) \subset [t_0, t_n]$, i.e., for
$j=0,1,\ldots,n-1$ and $i=j-m,j-m+1,\ldots,j$.
\end{problem}

Notice that if the coefficents $b_k^{(i,j)}$ are already known, it is possible
to use the geometric algorithm proposed in~\cite{WCh2020}
to compute any B-spline function in linear time with respect to its degree.
By using the computed values of B-spline functions, one can efficiently
evaluate a point on a B-spline curve.
Additionally, the knowledge of the coefficients of a B-spline function
can allow to perform more operations, such as differentiation or integration,
analytically, e.g.,
when solving differential or integral equations
using numerical methods.

The adjusted Bernstein-B\'{e}zier coefficients of B-spline functions
satisfy some easy to prove properties.
\begin{theorem}\label{T:BB-NonNeg}
For $u \in [t_j, t_{j+1})$ $(j=0,1,\ldots,n-1)$, the coefficients
$b_{k,m}^{(i,j)}$ ($k=0,1,\ldots,m$) of the adjusted Bernstein-B\'{e}zier
representation of the B-spline function $N_{mi}$ (cf.~Eq.~\eqref{E:BSplineABB})
are non-negative.
\end{theorem}
%
%

\begin{theorem}\label{T:BB-Partition1}
For $u \in [t_j, t_{j+1})$ $(j=0,1,\ldots,n-1)$, the following relation holds:
$$
\sum_{i=j-m}^j b_{k,m}^{(i,j)} = 1 \qquad (k=0,1,\ldots,m),
$$
where $b_{k,m}^{(i,j)}$ are the adjusted Bernstein-B\'{e}zier coefficients of
$N_{mi}$ (cf.~\eqref{E:BSplineABB}).
\end{theorem}

A task similar to Problem~\ref{P:BSpline1} for the \textit{adjusted power basis form}
of the B-spline functions has been considered in~\cite{FChPhD}.

Let the adjusted power basis form of the B-spline function $N_{mi}$ over
a single non-empty knot span $[t_j, t_{j+1}) \subset [t_0, t_n]$
($j=i,i+1,\ldots,i+m$) be
\begin{equation}\label{E:BSplineAPB}
N_{mi}(u) = \sum_{k=0}^m a_k^{(i,j)} (u - t_j)^k\qquad (t_j \leq u < t_{j+1}),
\end{equation}
with $a_k^{(i,j)} \equiv a_{k,m}^{(i,j)}$.

Explicit expressions for the adjusted power basis coefficients of $N_{mi}$ have
been given in~\cite{LiuWang02}, and the result can be adapted for the adjusted
Bernstein-B\'{e}zier form. The serious drawback of this approach, however, is
high complexity, which greatly limits the use of this result in computational
practice.

Another algorithm for finding the adjusted power basis coefficients of a spline
\begin{equation}\label{E:Spline}
s(t) := \sum_{i=-m}^{n-1} c_i N_{mi}(t).
\end{equation}
over a knot span $[t_j, t_{j+1})$ in $O(m^2)$ time
can be found in~\cite[\S1.3.2]{Dierckx93}.
By setting
$$ 
c_k := 
\left\{
\begin{array}{ll}
 1 & (k=i),\\
 0 & \mbox{otherwise}
\end{array}
\right.
$$
(see~\eqref{E:Spline}),
one can find the coefficients of
$N_{mi}$ over $[t_j, t_{j+1})$ in $O(m^2)$ time.
In total, to find the adjusted power basis coefficients
over $[t_j, t_{j+1})$ for all B-spline functions $N_{mi}$ such that 
$j-m \leq i \leq j$, one has to do $O(m^3)$ operations. Let us assume that there
are $n_e$ non-empty knot spans $[t_j, t_{j+1})$ such that $j=0,1,\ldots,n-1$. 
To find the coefficients of all B-spline functions over all non-empty knot spans
$[t_j, t_{j+1})$ for $j=0,1,\ldots,n-1$, one would need to perform $O(n_e m^3)$
operations.

With a similar approach, one can find the Bernstein-B\'{e}zier coefficients 
of $N_{mi}$ over the knot span $[t_j, t_{j+1})$. One can check that
\begin{equation*}
b_k^{(i,j)} = \dfrac{(m-k)!}{m!} N_{mi}^{(k)}(t_j)
                  - \sum_{\ell = 0}^{k-1} (-1)^{k-\ell} \binom{k}{\ell}
                                     b_{\ell}^{(i,j)} \qquad (k=0,1,\ldots,m)
\end{equation*}
(cf.~\cite[{Eq.~(5.25)}]{Farin2002} and \cite[{Theorem 4.1}]{LW2011}).
Just as in the case of the adjusted power basis, the Bernstein-B\'{e}zier coefficients
$b_k^{(i,j)}$ of $N_{mi}$ over $[t_j, t_{j+1})$ can be found in $O(m^2)$ time. 
In total, to find these coefficients for all B-spline functions over all
non-empty knot spans $[t_j, t_{j+1})$ for $j=0,1,\ldots,n-1$, it is required to
perform $O(n_e m^3)$ operations.

The approach given in \cite{Sablonniere78} and \cite{Boehm77} serves to convert 
a B-spline curve segment into a B\'{e}zier curve. It can be adapted to give an
algorithm with $O(m^3)$ complexity for finding the adjusted Bernstein-B\'{e}zier
coefficients $b_k^{(i,j)}$ of a single basis function $N_{mi}$. Doing so for 
each B-spline function in each non-empty knot span takes $O(n_e m^4)$ 
operations.

The Bernstein-B\'{e}zier coefficients in each knot span of a linear combination
of B-spline functions can also be found by using
the knot insertion method (see~\cite{Boehm80, CLR80}) in the following way.
After inserting the knots so that each has multiplicity $m+1$,
the basis functions become Bernstein polynomials and the coefficients
can be easily read. If all inner knots have multiplicity $1$,
the complexity of such procedure is $O(nm^2)$.

This does, however, return only the coefficients of one combination
of B-spline basis functions.
In order to find the coefficients of all B-spline basis functions
$$
N_{m, -m},N_{m, -m+1},\ldots, N_{m, n-1},
$$
one needs to do this procedure for each basis function
(i.e., with the coefficient $1$ for the selected basis function and $0$ otherwise).
This can be done in $O(nm^3)$ total time.

A generalization of B-splines which is currently gaining prominence are
multi-degree B-splines.
In~\cite{BC2021, BC2022, TS2020}, methods of finding
the coefficients of multi-degree B-splines are given,
which can be adapted for standard (uniform-degree) B-splines.
In particular, the method given in \cite{BC2022},
which is based on reverse knot insertion,
is shown to be numerically stable.

If there are sufficiently many computationally simple recurrence relations
for the Bernstein-B\'{e}zier coefficients of the B-spline functions
over multiple knot spans, one can instead use them to efficiently find each of
the coefficients. Over the course of this paper, they will be derived
from a new differential-recurrence relation for the B-spline functions.

\begin{remark}\label{R:BSplineAssumptions}
In the sequel, we assume that no inner knot $t_1, t_2, \ldots, t_{n-1}$ 
has multiplicity greater than $m$. This guarantees the B-spline functions'
continuity in $(t_0, t_n)$.
\end{remark}

The assumption regarding the multiplicity of the inner knots is very common
and intuitive, as it guarantees the continuity of a B-spline curve.
It was used, e.g., in~\cite{Dierckx93, LiuWang02} and \cite[{\S 3}]{Wozny2014}. 

Let us suppose that there are $n_e$ non-empty knot spans $[t_j, t_{j+1})$ such
that $0\leq j\leq n-1$. One of the main goals of the paper is to give a recursive
way of computing all $O(n_e m^2)$ coefficients $b_k^{(i,j)}$
(cf.~\eqref{E:BSplineABB}) of the
B-spline functions in $O(n_e m^2)$ time, assuming that all boundary knots are
coincident. 

The possible applications of this result can be as follows. Once the 
adjusted Bernstein-B\'{e}zier coefficients $b_k^{(i,j)}$ are known, each point 
on a B-spline curve \p{S},
\begin{equation}\label{E:BSplineCurve}
\p{S}(u):=\sum_{i=-m}^{n-1}N_{mi}(u)\p{W}_i
\qquad(t_0\leq u\leq t_n;\ \p{W}_i\in{\mathbb E}^d),	
\end{equation}
can be computed in $O(m^2 + md)$ time using the geometric algorithm proposed
recently by the authors in~\cite{WCh2020}. If there are $N$ such points on $M$
curves (each with the same knots) --- a situation closely related to rendering
a tensor product B-spline surface --- the total complexity is
$O(n_e m^2 + Nm^2 + MNmd)$, compared to $O(MNm^2d)$ when using the de Boor-Cox
algorithm. Performed experiments confirm that the new method is faster than
the de~Boor-Cox algorithm even for low $M \approx 2, 3$. Using a similar 
approach, one can also compute the value of any $N_{mi}$ in $O(m)$ time.

Additionally, if the Bernstein-B\'{e}zier coefficients of the B-spline basis functions are known,
one can use them to convert a B-spline curve over one non-empty knot span
to a B\'{e}zier curve:
$$
\p{S}(u) = \sum_{k=0}^m \Big( \sum_{i=j-m}^{j}  b^{(i,j)}_k \p{W}_i \Big)
                         B^m_k\Big (\dfrac{u - t_j}{t_{j+1} - t_j} \Big) \qquad (t_j \leq u < t_{j+1}).
$$

\section{New differential-recurrence relation for B-spline functions}
                                                      \label{S:BSplineDiffRec}

Using the recurrence relation~\eqref{E:BSplineRec} which connects B-spline
functions of consecutive degrees, one can find a recurrence relation which
is satisfied by their coefficients in the chosen basis. 
It is, however, not optimal as the recurrence scheme
is analogous to the one used in the de Boor-Cox algorithm.

A new differential-recurrence relation for the
B-spline functions of the same degree $m$ will be derived. We show that by using 
this result, it is possible to find all the Bernstein-B\'{e}zier coefficients
faster (see Section~\ref{S:BSplineBezier}).


Using equations~\eqref{E:BSplineRec} and~\eqref{E:BSplineDeriv}, one can derive
new differential-recurrence relations for the B-spline functions of the same
degree. For example, this result can be used to efficiently compute the
coefficients of the $N_{mi}$ functions (which are polynomial in each of the
\textit{knot spans}) in an adjusted Bernstein-B\'{e}zier or power basis.

\begin{theorem}\label{T:BSplineDiffRec}
Let
$t_{-m} = t_{-m+1} = \ldots = t_0 < t_1 < \ldots < t_{n-1} < t_n 
                                                  = t_{n+1} = \ldots = t_{n+m}$
(cf.~\eqref{E:KnotClones}). The following relations hold:
\begin{align}
  &m N_{m, -m}(u) + (t_1- u) N_{m,-m}'(u) = 0,
    \label{E:BSplineSimple1}&\\
  &N_{mi}(u) + \dfrac{t_i - u}{m} N_{mi}'(u) =
      \dfrac{t_{m+i+1} - t_i}{t_{m+i+2} - t_{i+1}}
      \left(N_{m,i+1}(u) + \dfrac{t_{m+i+2} - u}{m} N_{m,i+1}'(u)\right)
  \label{E:BSplineDiffRecBase}&\\
  \nonumber & \qquad \qquad \qquad \qquad \qquad \qquad
              \qquad \qquad \qquad \qquad \qquad(i=-m,-m+1,\ldots,n-2),&\\
  &m N_{m,n-1}(u) + (t_{n-1} - u) N_{m,n-1}'(u) = 0.&
  \label{E:BSplineSimple2}
\end{align}
\end{theorem}
\begin{proof}
Equations~\eqref{E:BSplineSimple1} and~\eqref{E:BSplineSimple2} follow easily
from equations~\eqref{E:BSplineRec} and~\eqref{E:BSplineDeriv}, respectively. 
The relation~\eqref{E:BSplineDiffRecBase} follows directly from taking the
expression for $N_{m+1,i}$ from Eq.~\eqref{E:BSplineRec} and differentiating
it, then equating it with the expression for $N_{m+1,i}'$ given in
Eq.~\eqref{E:BSplineDeriv}.
\end{proof}

Theorem~\ref{T:BSplineDiffRec} can be used to find a recurrence relation
satisfied by the adjusted Bernstein-B\'{e}zier coefficients of B-spline 
functions of the same degree, as will be shown in the next section.

\section{Recurrence relations for B-spline functions' coefficients in
                                          adjusted Bernstein-B\'{e}zier basis}
                                                       \label{S:BSplineBezier}

Assume that that the knot sequence is clamped, i.e.,
$$
t_{-m} = t_{-m+1} = \ldots = t_0 < t_1 < \ldots < t_n =
                                                    t_{n+1} = \ldots = t_{n+m}
$$
(cf.~\eqref{E:KnotClones}). 

For each knot span $[t_j, t_{j+1})$ ($j=0,1,\ldots,n-1$), one needs to find the
coefficients of $N_{mi}$ ($i=j-m,j-m+1,\ldots,j$) in the following adjusted
Bernstein-B\'{e}zier basis form:
$$
N_{mi}(u) = \sum_{k=0}^m b^{(i,j)}_k B^m_k(t) \qquad (t_j \leq u < t_{j+1}),
$$
where $b_k^{(i,j)} \equiv b_{k,m}^{(i,j)}$ and
\begin{equation}\label{E:TofU}
t \equiv t^{(j)}(u) := \dfrac{u - t_j}{t_{j+1} - t_j}
\end{equation}
(cf.~Eq.~\eqref{E:BSplineABB} and Problem~\ref{P:BSpline1}).
Additionally, then,
$u = (t_{j+1} - t_j)t + t_j$.

Certainly, $N_{mi}(u) \equiv 0$ if $u < t_i$ or $u > t_{m+i+1}$, which means 
that for a given knot span $[t_j, t_{j+1})$, one only needs to find the
coefficients of $N_{m,j-m}, N_{m,j-m+1},\ldots,N_{mj}$, as all coefficients of
other B-spline functions over this knot span are identical to zero. Thus, in 
each of $n$ knot spans, there are $m+1$ non-zero B-spline functions, each with
$m+1$ coefficients.

Solving Problem~\ref{P:BSpline1} requires computing $n (m+1)^2$ coefficients
$b_k^{(i,j)}$. In this section, it will be shown how to do it in $O(nm^2)$
time --- proportionally to the number of coefficients.
Theorem~\ref{T:BSplineDiffRec} serves as a foundation of the presented approach.
More precisely, the theorem will be used to construct recurrence relations
for the coefficients $b_k^{(i,j)}$ which allow solving Problem~\ref{P:BSpline1}
efficiently.

The results for particular cases will be presented in stages. 
In~\S\ref{SS:BSplineStage1}, explicit expressions for the coefficients of
$N_{mj}$ and $N_{m,j-m}$ over $[t_j, t_{j+1})$ ($j=0,1,\ldots,n-1$) will be
found. This will, in particular, cover the only non-trivial knot span for
$N_{m,n-1}$. In \S\ref{SS:BSplineStage2}, Eq.~\eqref{E:BSplineDiffRecBase} will
be applied to find the coefficients of $N_{mi}$ for $j=n-1,n-2,\ldots,0$ and
$i=j-1,j-2,\ldots,j-m+1$.

\begin{remark}
  In the sequel, it will be assumed that $u \in [t_j, t_{j+1})$.
  Thus, $t \in [0, 1)$ (cf.~\eqref{E:TofU}).
\end{remark}

\subsection{Stage 1}                                  \label{SS:BSplineStage1}

For $j=0,1,\ldots,n-1$, one can use Eq.~\eqref{E:BSplineRec} for $i=j$, along
with the fact that $N_{\ell,j+1} \equiv 0$ over $[t_j, t_{j+1})$
($\ell=m-1,m-2,\ldots,0$), to find that
\begin{equation*}
N_{mj}(u) = \dfrac{(u-t_j)^m}{\prod_{k=1}^m (t_{j+k} - t_j)} N_{0j}(u)
= \dfrac{(t_{j+1}-t_j)^{m-1}}{\prod_{k=2}^m (t_{j+k} - t_j)} B^m_m(t)
\qquad (u \in [t_j, t_{j+1})).
\end{equation*}
It means that
\begin{equation}\label{E:BSplineStage1Exp}
\left\{\begin{array}{ll}
  b_k^{(j,j)} = 0 & (k = 0,1,\ldots,m-1),\\
  b_m^{(j,j)} = \dfrac{(t_{j+1}-t_j)^{m-1}}{\prod_{k=2}^m (t_{j+k} - t_j)},&
\end{array}\right.
\end{equation}
where $0 \leq j \leq n-1$.

Using the same approach for $N_{m,j-m}$ over $[t_j, t_{j+1})$  gives
\begin{equation*}
N_{m,j-m}(u) = \dfrac{(t_{j+1} - t_j)^{m-1}}{\prod_{k=2}^{m}(t_{j+1} -
                                                          t_{j+1-k})} B^m_0(t)
                                                          \qquad (u \in [t_j, t_{j+1})).
\end{equation*}
The coefficients $b_k^{(j-m,j)}$ ($k=0,1,\ldots,m$) are thus given by the
following formula:
\begin{equation}\label{E:BSplineStage1LastExp}
\left\{\begin{array}{ll}
  b_0^{(j-m,j)} = \dfrac{(t_{j+1} - t_j)^{m-1}}
                        {\prod_{k=2}^{m}(t_{j+1} - t_{j+1-k})},&\\
  b_k^{(j-m,j)} = 0 & (k = 1,2,\ldots,m),
\end{array}\right.
\end{equation}
where $0 \leq j \leq n-1$. The adjusted Bernstein-B\'{e}zier coefficients of
$N_{mj}$ and $N_{m,j-m}$ over the knot span $[t_j, t_{j+1})$
(cf.~Eq.~\eqref{E:BSplineABB}) have been found for $j=0,1,\ldots,n-1$.

In the sequel, the following observation will be of use.
\begin{remark}\label{R:LastColZero}
Note that
$$
N_{m,n-1}(t_n) = \dfrac{(t_n-t_{n-1})^{m-1}}
                       {\prod_{k=1}^{m-1} (t_{n+k} - t_{n-1})} B^m_m(1) = 1,
$$
since $t_n = t_{n+1} = \ldots = t_{n+m}$. The B-spline functions have the
partition of unity property and are non-negative
(cf.~Theorem~\ref{T:BSplineNonNegSum1}), it is thus clear that
$$
N_{mi}(t_n) = 0 \qquad (i=-m,-m+1,\ldots,n-2).
$$
Similarly,
$$
N_{m,-m}(t_0)=\dfrac{(t_1 - t_0)^m}{\prod_{k=1}^{m}(t_1-t_{1-k})}B^m_0(0)=1,
$$
since $t_{-m} = t_{-m+1} = \ldots = t_0$. It follows that $N_{mi}(t_0) = 0$ for
$i=-m+1,-m+2,\ldots,0$.
\end{remark}

\subsection{Stage 2}                                  \label{SS:BSplineStage2}

To compute the coefficients of all functions $N_{mi}$ over knot spans
$[t_j, t_{j+1})$ such that $j=n-1,n-2,\ldots,0$ and $i=j-1,j-2,\ldots,j-m+1$,
Eq.~\eqref{E:BSplineDiffRecBase} will be used. The following identity will be
useful when operating on Eq.~\eqref{E:BSplineDiffRecBase}:
\begin{equation}\label{E:BSplineDerivVar}
\Big(N_{mi}(u) \Big)' = \dfrac{d N_{mi}(u)}{du}
= \sum_{k=0}^m b^{(i,j)}_k \dfrac{d B^m_k(t)}{dt} \cdot \dfrac{d t}{du}
= (t_{j+1} - t_j)^{-1} \sum_{k=0}^m b^{(i,j)}_k \Big( B^m_k(t) \Big)',
\end{equation}
with $u \in [t_j, t_{j+1})$ (cf.~\eqref{E:TofU}).

Let
\begin{equation}\label{E:BSplineV}
v_i \equiv v_{mi} := \dfrac{t_{m+i+1} - t_i}{t_{m+i+2} - t_{i+1}}.
\end{equation}

Substituting the adjusted Bernstein-B\'{e}zier forms of $N_{mi}$ and $N_{m,i+1}$
in the knot span $[t_j, t_{j+1})$ and applying Eq.~\eqref{E:BSplineDerivVar} 
into Eq.~\eqref{E:BSplineDiffRecBase} gives
\begin{multline*}
  \sum_{k=0}^m b^{(i,j)}_k B^m_k(t) + \Big(\dfrac{t_i - t_j}{m (t_{j+1} - t_j)} 
  - \dfrac{t}{m}\Big) \sum_{k=0}^m b^{(i,j)}_k \left(B^m_k(t)\right)' = \\
  = v_i \left(\sum_{k=0}^m b^{(i+1,j)}_k B^m_k(t) + 
           \Big(\dfrac{t_{m+i+2} - t_j}{m (t_{j+1} - t_j)} - \dfrac{t}{m}\Big)
                      \sum_{k=0}^m b^{(i+1,j)}_k \left(B^m_k(t)\right)' \right).
\end{multline*}
After using identities~\eqref{E:BernsteinDeriv} and~\eqref{E:BernsteinuDeriv} 
and doing some algebra, one gets
\begin{multline*}
 \sum_{k=0}^{m} \left(l_{ki} b^{(i,j)}_{k-1} + d_{ki} b^{(i,j)}_k
                      + u_{ki} b^{(i,j)}_{k+1} \right) B^m_{k}(t) = \\ 
    = v_i \sum_{k=0}^{m} \Big(l_{k,m+i+2} b^{(i+1,j)}_{k-1} + d_{k,m+i+2}
                b^{(i+1,j)}_k + u_{k,m+i+2} b^{(i+1,j)}_{k+1} \Big) B^m_{k}(t),
\end{multline*}
where
$$
l_{kr} := k (t_{j+1} - t_r), \quad
d_{kr} := (m-k) (t_{j+1} - t_r) + k (t_r - t_j), \quad 
u_{kr} := (m-k) (t_r - t_j).
$$
Matching the coefficients of Bernstein polynomials on both sides gives a system
of $m+1$ equations of the form:
\begin{equation}\label{E:BSplineBezCase1}
\left\{
\begin{array}{l}
    (t_{j+1} - t_i) b^{(i,j)}_0 + (t_i - t_j) b^{(i,j)}_1 = v_i
        \Big((t_{j+1} - t_{m+i+2}) b^{(i+1,j)}_0 + (t_{m+i+2} - t_j)
                                                   b^{(i+1,j)}_1\Big),\\[3ex]
    l_{ki} b^{(i,j)}_{k-1} + d_{ki} b^{(i,j)}_k + u_{ki} b^{(i,j)}_{k+1}
    =v_i \Big(l_{k,m+i+2} b^{(i+1,j)}_{k-1} + d_{k,m+i+2} b^{(i+1,j)}_k
                                  + u_{k,m+i+2} b^{(i+1,j)}_{k+1}\Big)\\[2ex]
                                   \hspace*{\fill} (k=1,2,\ldots,m-1),\\[3ex]
   (t_{j+1} - t_i) b^{(i,j)}_{m-1} + (t_i - t_j) b^{(i,j)}_m
    = v_i \Big((t_{j+1} - t_{m+i+2}) b^{(i+1,j)}_{m-1} + (t_{m+i+2} - t_j)
                                                           b^{(i+1,j)}_m\Big).
\end{array}
\right.
\end{equation}

\begin{theorem}\label{T:BSplineBezCase1}
For $j=0,1,\ldots,n-1$ and $i=j-1,j-2,\ldots,j-m+1$, assuming that the
coefficients $b^{(i+1,j)}_k$ $(0\leq k\leq m$) are known, the values
$b^{(i,j)}_0, b^{(i,j)}_1, \ldots, b^{(i,j)}_m$ satisfy a first-order
non-homogeneous recurrence relation
\begin{equation}\label{E:BSplineBezCase1Th}
(t_{j+1} - t_i) b^{(i,j)}_k + (t_i - t_j) b^{(i,j)}_{k+1} = A(m,i,j,k)
                                                    \qquad (k=0,1,\ldots,m-1),
\end{equation}
where
$$
A(m,i,j,k) := v_i \Big((t_{j+1} - t_{m+i+2}) b^{(i+1,j)}_k +
                                    (t_{m+i+2} - t_j) b^{(i+1,j)}_{k+1} \Big)
$$
(cf.~\eqref{E:BSplineV}).
\end{theorem}
\begin{proof}
Base case ($k = 0$ and $k=m$): the relation holds and is presented in the first
and the last equations of the system~\eqref{E:BSplineBezCase1}.

Induction step ($k \rightarrow k+1$): the $(k+2)$th equation in the
system~\eqref{E:BSplineBezCase1} is
\begin{multline*}
l_{k+1,i} b^{(i,j)}_k + d_{k+1,i} b^{(i,j)}_{k+1} + u_{k+1,i} b^{(i,j)}_{k+2}=\\
 = v_i \Big(l_{k+1,m+i+2} b^{(i+1,j)}_k + d_{k+1,m+i+2} b^{(i+1,j)}_{k+1}
                                         + u_{k+1,m+i+2} b^{(i+1,j)}_{k+2}\Big).
\end{multline*}
Subtracting sidewise the induction assumption scaled by
$\dfrac{l_{k+1,i}}{(t_{j+1} - t_i)} = k+1$ gives, after some algebra,
\begin{multline*}
(t_{j+1} - t_i) b^{(i,j)}_{k+1} + (t_i - t_j) b^{(i,j)}_{k+2}
               = v_i \Big((t_{j+1} - t_{m+i+2}) b^{(i+1,j)}_{k+1}
                             + (t_{m+i+2} - t_j) b^{(i+1,j)}_{k+2}\Big),
\end{multline*}
which concludes the proof.
\end{proof}

From Theorem~\ref{T:BSplineBezCase1}, it follows that there are $m$ independent
equations in the system~\eqref{E:BSplineBezCase1}, as one of them is redundant.
One thus needs an initial value to find the values of all
$b^{(i,j)}_0, b^{(i,j)}_1, \ldots, b^{(i,j)}_m$
using the recurrence relation~\eqref{E:BSplineBezCase1Th}.

If $j=n-1$, Remark~\ref{R:LastColZero} can be used to find that
$$
N_{mi}(t_n) = b_m^{(i,n-1)} = 0 \qquad (i=n-2,n-3,\ldots,n-m).
$$
In this case, the recurrence relation given in Theorem~\ref{T:BSplineBezCase1}
simplifies to
$$
(t_n - t_i) b^{(i,n-1)}_k = (t_{n-1} - t_i) b^{(i,n-1)}_{k+1} +
                                      v_i (t_n - t_{n-1}) b^{(i+1,n-1)}_{k+1}.
$$
It means that, for $i=n-2,n-3,\ldots,n-m$, the following relation holds:
\begin{equation}\label{E:BSplineStage3Last}
\left\{
\begin{array}{ll}
  b_m^{(i,n-1)} = 0,&\\
  b^{(i,n-1)}_k = \dfrac{t_{n-1} - t_i}{t_n - t_i} b^{(i,n-1)}_{k+1} +
                    \dfrac{t_n - t_{n-1}}{t_n - t_{i+1}} b^{(i+1,n-1)}_{k+1}
                                                         & (k=m-1,m-2,\ldots,0).
\end{array}
\right.
\end{equation}

For $i=n-2,n-3,\ldots,n-m$, assuming that the coefficients $b_k^{(i+1,n-1)}$ are
known ($k=1,2,\ldots,m$), Eq.~\eqref{E:BSplineStage3Last} has an
explicit solution
\begin{equation}\label{E:BSplineStage3LastExp}
\left\{
\begin{array}{ll}
  b_m^{(i,n-1)} = 0, & \\
  b_k^{(i,n-1)} = \dfrac{t_n - t_{n-1}}{t_n - t_{i+1}}
                   {\displaystyle \sum_{\ell=0}^{m-k-1}}
                     \Big(\dfrac{t_{n-1}-t_i}{t_n - t_i}\Big)^{\ell}
                     b_{k+1+\ell}^{(i+1,n-1)}
                     & (k=0,1,\ldots,m-1).
\end{array}
\right.
\end{equation}

To find the initial value, if $j < n-1$ and $i=j-1,j-2,\ldots,j-m+1$,
the right continuity condition will be used, i.e.,
$$
N_{mi}(t_{j+1}^-) = N_{mi}(t_{j+1}^+).
$$
More precisely,
$$
N_{mi}(t_{j+1}^-) = \sum_{k=0}^m b^{(i,j)}_k B^m_k(1) = b^{(i,j)}_m
$$
and
$$
N_{mi}(t_{j+1}^+) = \sum_{k=0}^m b_k^{(i, j+1)} B^m_k(0) = b_0^{(i, j+1)},
$$
which gives the relation
\begin{equation}\label{E:InitVal}
b^{(i,j)}_m = b_0^{(i,j+1)}.
\end{equation}

This completes the recurrence scheme for $j=n-2,n-3,\ldots,0$ and
$i=j-1,j-2,\ldots,j-m+1$:
\begin{equation}\label{E:BSplineStage3}
\left\{
\begin{array}{l}
   b_m^{(i,j)} = b_0^{(i,j+1)},\\
   b_k^{(i,j)} = \dfrac{t_j - t_i}{t_{j+1} - t_i} b_{k+1}^{(i,j)}
               + \dfrac{v_i}{t_{j+1} - t_i} \Big(
                   (t_{j+1} - t_{m+i+2}) b_k^{(i+1,j)}
                 + (t_{m+i+2} - t_j) b_{k+1}^{(i+1,j)} \Big)\\
   \qquad \qquad \qquad \qquad \qquad \qquad \qquad \qquad \qquad
   \qquad \qquad \qquad (k=m-1,m-2,\ldots,0).
\end{array}
\right.
\end{equation}

From Eq.~\eqref{E:BSplineStage3} follows an explicit formula for the 
coefficients $b^{(i,j)}_k$ $(0\leq k\leq m$), assuming that the coefficients
$b_0^{(i,j+1)}$ and $b^{(i+1,j)}_k$ ($k=0,1,\ldots,m$) are known:
\begin{equation}\label{E:BSplineStage3Exp}
b^{(i,j)}_k = \Big(\dfrac{t_j - t_i}{t_{j+1} - t_i}\Big)^{m-k}
                  b_0^{(i,j+1)}
              + \sum_{\ell = 0}^{m-k-1}
                  \Big(\dfrac{t_j - t_i}{t_{j+1} - t_i}\Big)^{\ell}
                            \dfrac{v_i}{t_{j+1} - t_i} q_{k + \ell},
\end{equation}
where
$$
 q_{\ell} := (t_{j+1} - t_{m+i+2}) b_{\ell}^{(i+1,j)}
                               + (t_{m+i+2}-t_j) b_{\ell+1}^{(i+1,j)},
$$ 
and $0 \leq j \leq n-2$, $j-m+1 \leq i \leq j-1$.

The coefficients of $N_{mi}$ have been found for $j=0,1,\ldots,n-1$
and $i=j-1,j-2,\ldots,j-m+1$.

\subsection{Recurrence scheme and implementation}

The results presented in~\S\ref{SS:BSplineStage1} and~\S\ref{SS:BSplineStage2}
can be combined to prove the following theorem.
\begin{theorem}\label{T:BSplineBezierBig}
Let us assume that
$$
t_{-m} = t_{-m+1} = \ldots = t_0 < t_1 < \ldots < t_{n-1} <
                                              t_n = t_{n+1} = \ldots = t_{n+m}
$$
(cf.~\eqref{E:KnotClones}). The $n(m+1)^2$ adjusted Bernstein-B\'{e}zier
coefficients $b_k^{(i,j)}$ of the B-spline functions $N_{mi}$ over each knot 
span $[t_j, t_{j+1})$ (cf.~\eqref{E:BSplineABB}), for $j=0,1,\ldots,n-1$,
$i=j-m,j-m+1,\ldots,j$ and $k=0,1,\ldots,m$, can be computed in the 
computational complexity $O(nm^2)$ in the following way:
\begin{enumerate}
\itemsep1ex

\item For $j=0,1,\ldots,n-1$ and $k=0,1,\ldots,m$, the coefficients
      $b_k^{(j,j)}$ and $b_k^{(j-m,j)}$ are given explicitly in
      equations~\eqref{E:BSplineStage1Exp} and~\eqref{E:BSplineStage1LastExp},
      respectively.

\item For $j = n-1$, $i = n-2,n-3,\ldots,n-m$ and $k=m,m-1,\ldots,0$,
      the coefficients $b_k^{(i,n-1)}$ ($k=0,1,\ldots,m$) are computed by the
      recurrence relation~\eqref{E:BSplineStage3Last} (for their explicit forms,
      see~\eqref{E:BSplineStage3LastExp}).
    
\item For $j=n-2,n-3,\ldots,0$, $i=j-1,j-2,\ldots,j-m+1$ and $k=m,m-1,\ldots,0$,
      the coefficients $b_k^{(i,j)}$ are computed by the recurrence
      relation~\eqref{E:BSplineStage3} (for their explicit forms,
      see~\eqref{E:BSplineStage3Exp}).

\end{enumerate}
\end{theorem}

\begin{example}\label{Ex:BSplineBez}
Let us set $m := 3$, $n := 5$. Let the knots be
$$
\begin{array}{c|c|c|c|c|c|c|c|c|c|c|c}
t_{-3} & t_{-2} & t_{-1} & t_0 & t_1 & t_2 & t_3 & t_4 &
                                           t_5 & t_6 & t_7 & t_8 \\ \hline
0 & 0 & 0 & 0 & 3 & 5 & 6 & 9 & 10 & 10 & 10 & 10
\end{array}.
$$
Figure~\ref{F:BSplineBezScheme} illustrates the approach to computing all
necessary adjusted Bernstein-B\'{e}zier coefficients of B-spline functions, 
given in Theorem~\ref{T:BSplineBezierBig}. Arrows denote recursive dependence.
Diagonally striped squares are computed using Eq.~\eqref{E:BSplineStage1Exp}.
Horizontally striped squares are computed using
Eq.~\eqref{E:BSplineStage1LastExp}.
White squares are computed using either the
recurrence~\eqref{E:BSplineStage3Last} (for $u \in [t_4, t_5)$)
or~\eqref{E:BSplineStage3} (for $u < t_4$).
\end{example}

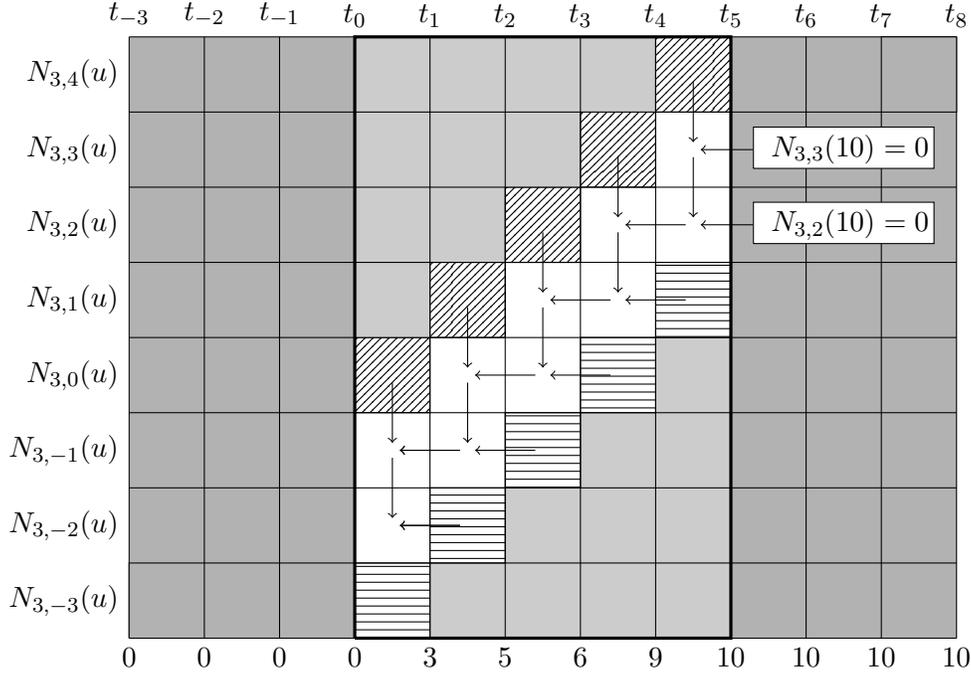
\begin{figure*}[ht!]
\centering
\begin{tikzpicture}
\foreach \col in {0,...,7}
 {
 \fill[black!20!white] (0, \col) -- (\col, \col) -- (\col, \col + 1) --
                       (0, \col + 1) -- (0, \col);
 \draw[pattern=north east lines] (\col, \col) -- (\col + 1, \col) --
                      (\col + 1, \col + 1) -- (\col, \col + 1) -- (\col, \col);
 \fill[black!20!white] (\col + 4, \col) -- (11, \col) -- (11, \col + 1) --
                       (\col + 4, \col + 1) -- (\col + 4, \col);
 \draw[pattern=horizontal lines] (\col+3, \col) -- (\col+4, \col) --
                      (\col+4, \col+1) -- (\col+3, \col+1) -- (\col+3, \col);
 }
 \fill[black!30!white] (0, 0) -- (3, 0) -- (3, 8) -- (0, 8) -- (0, 0);
 \fill[black!30!white] (8, 0) -- (11, 0) -- (11, 8) -- (8, 8) -- (8, 0);


 \foreach \row in {0,...,11}
 \draw (\row, 0) -- (\row, 8);
\foreach \col in {0,...,8}
 \draw (0, \col) -- (11, \col);

 \node[left] at (0, 0.5) {$N_{3, -3}(u)$};
 \node[left] at (0, 1.5) {$N_{3, -2}(u)$};
 \node[left] at (0, 2.5) {$N_{3, -1}(u)$};
 \node[left] at (0, 3.5) {$N_{3, 0}(u)$};
 \node[left] at (0, 4.5) {$N_{3, 1}(u)$};
 \node[left] at (0, 5.5) {$N_{3, 2}(u)$};
 \node[left] at (0, 6.5) {$N_{3, 3}(u)$};
 \node[left] at (0, 7.5) {$N_{3, 4}(u)$};

 \node[below] at (0,0) {$0$};
 \node[below] at (1,0) {$0$};
 \node[below] at (2,0) {$0$};
 \node[below] at (3, 0) {$0$};
 \node[below] at (4,0) {$3$};
 \node[below] at (5,0) {$5$};
 \node[below] at (6,0) {$6$};
 \node[below] at (7,0) {$9$};
 \node[below] at (8,0) {$10$};
 \node[below] at (9,0) {$10$};
 \node[below] at (10,0) {$10$};
 \node[below] at (11, 0) {$10$};

 \foreach \row in {-3,-2, -1, 1, 2, 3, 4, 6, 7, 8}
  {
    \node[above] at (\row + 3, 8) {$t_{\row}$};
  }

 \draw[very thick] (3, 0) -- (8, 0) -- (8, 8) -- (3, 8) -- (3, 0);
 \node[above] at (3, 8) {$t_{0}$};
 \node[above] at (8, 8) {$t_{5}$};

 \foreach \row in {3.5, 4.5, 5.5, 6.5, 7.5}
 {
   \draw[->] (\row, \row - 0.1) -- (\row, \row - 0.9);
   \draw[->] (\row, \row - 1.1) -- (\row, \row - 1.9);
 }
 \draw[->] (4.4, 2.5) -- (3.6, 2.5);
 \draw[->] (4.4, 1.5) -- (3.6, 1.5);
 \foreach \row in {3.5, 4.5, 5.5, 6.5}
 {
   \draw[->] (\row + 0.9, \row - 1) -- (\row + 0.1, \row - 1);
   \draw[->] (\row + 0.9, \row - 2) -- (\row + 0.1, \row - 2);
 }
 \draw[->] (9.2, 5.5) -- (7.6, 5.5);
 \draw[->] (9.2, 6.5) -- (7.6, 6.5);

 \fill[white] (8.3, 5.25) -- (10.7, 5.25) -- (10.7, 5.8) -- (8.3, 5.8) -- 
                                                            (8.3, 5.25); 
 \draw (8.3, 5.25) -- (10.7, 5.25) -- (10.7, 5.8) -- (8.3, 5.8) -- (8.3, 5.25); 
 \node[right] at (8.4, 5.5) {$N_{3,2}(10) = 0$};
 \fill[white] (8.3, 6.25) -- (10.7, 6.25) -- (10.7, 6.8) -- (8.3, 6.8) -- 
                                                            (8.3, 6.25); 
 \draw (8.3, 6.25) -- (10.7, 6.25) -- (10.7, 6.8) -- (8.3, 6.8) -- (8.3, 6.25); 
 \node[right] at (8.4, 6.5) {$N_{3,3}(10) = 0$};

 \end{tikzpicture}
\caption{An illustration of Example~\ref{Ex:BSplineBez}.}
\label{F:BSplineBezScheme}
\end{figure*}

\subsubsection{Implementation}
 
Algorithm~\ref{A:BSplineBezForm} implements the approach proposed in
Theorem~\ref{T:BSplineBezierBig}. This algorithm returns a sparse array 
$B\equiv B[0..n-1,-m..n-1,0..m]$, where
$$
B[j,i,k]=b^{(i,j)}_k\qquad (0\leq j<n,\; -m\leq i<n,\; 0\leq k\leq m)
$$
(cf.~\eqref{E:BSplineABB}).

For each of the $n$ knot spans, one has to compute the coefficients of $m+1$
functions ($n(m+1)^2$ coefficients in total). Computing all coefficients of one
B-spline function in a given knot span requires $O(m)$ operations. In total,
then, the complexity of Algorithm~\ref{A:BSplineBezForm} is $O(nm^2)$ ---
giving the optimal $O(1)$ time per coefficient.

\begin{algorithm}[ht!]
\caption{Computing the coefficients of the adjusted Bernstein-B\'{e}zier form
         of the B-spline functions}\label{A:BSplineBezForm}
\begin{algorithmic}[1]
\Procedure {BSplineBBF}{$n, m, [t_{-m}, t_{-m+1}, \ldots, t_{n+m}]$}
\State $B \gets \texttt{SparseArray[0..n-1, -m..n-1, 0..m](fill=0)}$
\For {$j \gets 0,n-1$}
  \State $B[j, j, m] \gets \dfrac{(t_{j+1}-t_j)^{m-1}}
                                 {\prod_{k=2}^m(t_{j+k}-t_j)}$
  \State $B[j,j-m,0] \gets \dfrac{(t_{j+1} - t_j)^{m-1}}
                                 {\prod_{k=2}^{m}(t_{j+1} - t_{j+1-k})}$
\EndFor
\For {$i \gets n-2, n-m$}
  \For {$k \gets m-1,0$}
    \State $B[n-1, i, k] \gets\dfrac{t_{n-1}-t_i}{t_n - t_i}\cdot B[n-1,i,k+1] +
                \dfrac{t_n-t_{n-1}}{t_n - t_{i+1}}\cdot B[n-1, i+1, k+1]$
  \EndFor
\EndFor
\For {$j \gets n-2, 0$}
  \For {$i \gets j-1, j-m+1$}
    \State $v \gets \dfrac{t_{m+i+1} - t_i}{t_{m+i+2} - t_{i+1}}$
    \State $B[j, i, m] \gets B[j+1,i,0]$
    \For {$k=m-1,0$}
      \State $B[j, i, k] \gets \dfrac{t_j-t_i}{t_{j+1}-t_i}\cdot
      B[j,i,k+1]+\dfrac{v}{t_{j+1}-t_i}\cdot \Big((t_{j+1}-t_{m+i+2})\cdot
      B[j,i+1,k] + (t_{m+i+2} - t_j)\cdot B[j,i+1,k+1]\Big)$
    \EndFor
  \EndFor
\EndFor
\State \Return $B$
\EndProcedure
\end{algorithmic}
\end{algorithm}

\section{Applications}\label{S:Applications}
\subsection{Fast computation of multiple points on multiple B-spline curves}
                                                     \label{SS:BSplineFastComp}

Let $u \in [t_j, t_{j+1})$ and $t := \dfrac{u-t_j}{t_{j+1}-t_j}$. By solving
Problem~\ref{P:BSpline1}, the Bernstein-B\'{e}zier coefficients of the B-spline
functions are found. A point on a B-spline curve~\eqref{E:BSplineCurve} can thus
be expressed as
$$
\p{S}(u) = \sum_{i=j-m}^{j}\Big(\sum_{k=0}^m b_k^{(i,j)} B^m_k(t)\Big)\p{W}_i.
$$

The inner sums
\begin{equation}\label{E:Def_pi}
p_i(u) := \sum_{k=0}^m b_k^{(i,j)} B^m_k(t)\equiv N_{mi}(u) 
                                                 \qquad (i=j-m,j-m+1,\ldots,j)
\end{equation}
can be treated as polynomial B\'{e}zier curves with control points
$b_k^{(i,j)} \in \mathbb{E}^1$ and thus can be computed using the geometric
algorithm given in~\cite{WCh2020} in total time $O(m^2)$ --- more precisely,
$O(m)$ per each of $m+1$ sums. It also means that --- when all the coefficients
$b_k^{(i,j)}$ are already known --- any B-spline function may be computed in 
linear time with respect to its degree. 

\begin{example}\label{E:BSplineExampleFun}
A comparison of the new method of evaluating B-spline functions and using
recurrence relation~\eqref{E:BSplineRec} has been done. The results have been
obtained on a computer with \texttt{Intel Core i5-6300U CPU} at \texttt{2.40GHz}
processor and \texttt{4GB} \texttt{RAM}, using \texttt{GNU C Compiler 11.2.0}
(single precision).

For each $n \in \{10, 15, 20, 25, 30, 35, 40, 45, 50\}$, and 
$m=3, 4,\ldots, 15$, a sequence of knots has been generated $100$ times. The knot
span lengths  $t_{j+1} - t_j \in [1/50, 1]$ ($j=0,1,\ldots,n-1;\; t_0 = 0$) have
been generated using the \texttt{rand()} C function. The boundary knots are
coincident. Then, $50 \cdot n + 1$ points such that 
$t_{j \ell} := t_j + \ell/50 \times (t_{j+1} - t_j)$ for $j=0,1,\ldots,n-1$ and
$\ell=0,1,\ldots,49$, with the remaining point being $t_{n0} \equiv t_n$, are
generated.

At each point $t_{j \ell} \in [t_j, t_{j+1})$, all $m+1$ B-spline functions
$N_{mi}$ ($i = j - m, j - m + 1, \ldots, j$) which do not vanish at $t_{j \ell}$
are evaluated using both algorithms. Due to the size of the table, the resulting
running times are available at 
\url{https://www.ii.uni.wroc.pl/~pwo/programs/BSpline-BF-Example-5-1.xlsx}.

The new method consistently performs faster than evaluating B-spline functions
using recurrence relation~\eqref{E:BSplineRec}. The new method reduced the 
running time for any dataset by $29$--$46\%$, while the total running time was
reduced by $45\%$. The source code in C which was used to perform the tests is
available at \url{https://www.ii.uni.wroc.pl/~pwo/programs/BSpline-BF.c}.
\end{example}

Note that the sums $p_i$ (cf.~\eqref{E:Def_pi}) do not depend on the control
points. Afterwards, computing a convex combination of $m+1$ points from
$\mathbb{E}^d$, i.e.,
$$
\p{S}(u) = \sum_{i=j-m}^{j} p_i(u) \p{W}_i \qquad (u \in [t_j, t_{j+1})),
$$
requires $O(md)$ arithmetic operations. Observe that these values may also be
computed using the geometric method proposed in~\cite[Algorithm 1.1]{WCh2020}.
In total, then, assuming that the Bernstein-B\'{e}zier coefficients of the
B-spline functions over each knot span $[t_j, t_{j+1})$ ($j=0,1,\ldots,n-1$) are
known, $O(m (m + d))$ arithmetic operations are required to compute a point
$\p{S}(u)$ ($u \in [t_j, t_{j+1})$) on a B-spline curve.

When it is required to compute the values of $\p{S}$ for many parameters
$u_0,u_1,\ldots,u_N$, one would have to perform $O(nm^2)$ arithmetic operations
to find the Bernstein-B\'{e}zier coefficients of the B-spline functions over 
each knot span and then do $O(m(m + d))$ operations for each of $N+1$ points 
that are to be computed. In total, the computational complexity of this approach
is $O(nm^2 + Nm(m+d))$.

Due to the fact that the sums $p_i$ (cf.~\eqref{E:Def_pi}) do not depend on the
control points, they can be used for computing a point on multiple B-spline
curves, all of degree $m$, with the same knots.
\begin{problem}\label{P:BSpline3}
For $M$ B-spline curves $\p{S}_0, \p{S}_1, \ldots, \p{S}_{M-1}$ with the knots
$$
t_{-m} = t_{-m+1} = \ldots = t_0 < t_1 < \ldots < t_n =
                                                    t_{n+1} = \ldots = t_{n+m}.
$$
and the control points of $\p{S}_k$ being
$$
\p{W}_{k,-m},\p{W}_{k,-m+1},\ldots,\p{W}_{k,n-1}\in{\mathbb E^d} 
                                                     \qquad (k=0,1,\ldots,M-1),
$$
compute the value of each of the B-spline curves $\p{S}_k$ at points
$u_0, u_1,\ldots,u_{N-1}$ such that $t_0 \leq u_k \leq t_n$ for all
$k=0,1,\ldots,N-1$. More precisely, for $k=0,1,\ldots,M-1$ and
$\ell=0,1,\ldots,N-1$, compute all the points $\p{S}_k(u_{\ell})$.
\end{problem}
One can efficiently solve Problem~\ref{P:BSpline3} in the following way.
Using Algorithm~\ref{A:BSplineBezForm} allows to compute all the adjusted
Bernstein-B\'{e}zier coefficients of B-spline functions
(cf.~Problem~\ref{P:BSpline1}) in $O(nm^2)$ time. Now, one needs to compute the
values
$$
p_i(u_{\ell}) \qquad (\ell=0,1,\ldots,N-1,\;u_{\ell} \in [t_j, t_{j+1}),\;
                      i=j-m,j-m+1,\ldots,j)
$$
(cf.~\eqref{E:Def_pi}), which takes $O(Nm^2)$ time. Using these values,
computing
$$
\p{S}_k(u_{\ell}) = \sum_{i=j-m}^{j} p_i(u_{\ell}) \p{W}_{ki}
   \qquad (\ell=0,1,\ldots,N-1,\;k=0,1,\ldots,M-1,\;u_{\ell}\in [t_j, t_{j+1}))
$$
takes $O(MNmd)$ time (see \cite{WCh2020}). In total, then, the complexity of this
approach is $O(nm^2 + Nm^2 + NMmd)$, compared to the complexity of using the
de~Boor-Cox algorithm to solve Problem~\ref{P:BSpline3}, i.e., $O(NMm^2d)$.

Using the recurrence relation~\eqref{E:BSplineRec} (see~\cite[p.~55--57]{DeBoor1972})
to evaluate the B-spline basis functions in $O(Nm^2)$ and then compute
the linear combinations of control points in $O(NMmd)$ time gives a total
complexity of $O(Nm^2 + NMmd)$. While this may appear similar
to the complexity of the new method, a closer examination of the number
of floating-point operations shows that the new method
has $O(Nm)$ divisions, compared to $O(Nm^2)$
in the approach based on
the relation~\eqref{E:BSplineRec}.
In practice, this
saves $O(Nm^2)$ divisions (see Table~\ref{T:TableOps}).

A comparison of running times is given in Example~\ref{E:BSplineExample}.
The new algorithm is compared to executing the de Boor-Cox algorithm 
and to an alternative way of computing the B-spline functions based on 
the recurrence relation~\eqref{E:BSplineRec} (see~\cite[p.~55--57]{DeBoor1972})
and then evaluating the point in the same way as in the new method.

\begin{example}\label{E:BSplineExample}
Table~\ref{T:TableBS} shows the comparison between the running times of the de
Boor-Cox algorithm, an algorithm which computes the values of B-spline function
using the recurrence relation~\eqref{E:BSplineRec} and then computes the points,
and the new method described above and using Algorithm~\ref{A:BSplineBezForm}.

The results have been obtained on a computer with
\texttt{Intel Core i5-6300U CPU} at \texttt{2.40GHz} processor and \texttt{4GB}
\texttt{RAM}, using \texttt{GNU C Compiler 11.2.0} (single precision).

The following numerical experiments have been conducted. For fixed $n=20$ and
$d=2$, for each $M \in \{1, 5, 10, 20, 50, 100\}$ and $m \in \{3, 5, 7, 9, 11\}$,
a sequence of knots and control points has been generated $100$ times.
The control points $\p{W}_{ki} \in [-1, 1]^d$ 
$(i=-m,-m+1,\ldots,n-1,\; k=0,1,\ldots,M-1$) and the knot span lengths 
$t_{j+1} - t_j \in [1/50, 1]$ ($j=0,1,\ldots,n-1;\; t_0 = 0$) have been generated
using the \texttt{rand()} C function. The boundary knots are coincident. Each
algorithm is then tested using the same knots and control points. Each curve is
evaluated at $1001$ points which are $t_j + \ell/50 \times (t_{j+1} - t_j)$ for
$j=0,1,\ldots,n-1$ and $\ell=0,1,\ldots,49$, with the remaining point being 
$t_n$. Table~\ref{T:TableBS} shows the total running time of all 
$100 \times 1001 \times M$ curve evaluations for each method.

On average,
the method \texttt{eval splines} had $7.52$,
while \texttt{new method} had $7.24$ common digits
with the result of the
numerically stable
de Boor-Cox algorithm in single precision
(8 digits) computations.
\end{example}

\begin{table*}[ht!]\small
\begin{center}
\renewcommand{\arraystretch}{1.3}
\begin{tabular}{lcccc}
$M$ & $m$ & de~Boor-Cox & eval splines & new method \\ \hline
1 & 3 & \textbf{0.017} & 0.027 & 0.019 \\
1 & 5 & 0.033 & 0.046 & \textbf{0.030} \\
1 & 7 & 0.056 & 0.073 & \textbf{0.046} \\
1 & 9 & 0.090 & 0.110 & \textbf{0.067} \\
1 & 11 & 0.129 & 0.151 & \textbf{0.089} \\\hline
5 & 3 & 0.076 & 0.045 & \textbf{0.038} \\
5 & 5 & 0.161 & 0.077 & \textbf{0.061} \\
5 & 7 & 0.281 & 0.116 & \textbf{0.086} \\
5 & 9 & 0.445 & 0.161 & \textbf{0.115} \\
5 & 11 & 0.643 & 0.212 & \textbf{0.147} \\\hline
10 & 3 & 0.151 & 0.078 & \textbf{0.065} \\
10 & 5 & 0.323 & 0.116 & \textbf{0.097} \\
10 & 7 & 0.562 & 0.167 & \textbf{0.135} \\
10 & 9 & 0.890 & 0.224 & \textbf{0.174} \\
10 & 11 & 1.285 & 0.287 & \textbf{0.218} \\\hline
20 & 3 & 0.302 & 0.142 & \textbf{0.115} \\
20 & 5 & 0.645 & 0.194 & \textbf{0.171} \\
20 & 7 & 1.126 & 0.276 & \textbf{0.231} \\
20 & 9 & 1.775 & 0.350 & \textbf{0.293} \\
20 & 11 & 2.568 & 0.438 & \textbf{0.358} \\\hline
50 & 3 & 0.754 & 0.333 & \textbf{0.267} \\
50 & 5 & 1.612 & 0.428 & \textbf{0.391} \\
50 & 7 & 2.810 & 0.579 & \textbf{0.517} \\
50 & 9 & 4.450 & 0.729 & \textbf{0.648} \\
50 & 11 & 6.423 & 0.889 & \textbf{0.781} \\\hline
100 & 3 & 1.510 & 0.655 & \textbf{0.524} \\
100 & 5 & 3.225 & 0.822 & \textbf{0.760} \\
100 & 7 & 5.622 & 1.099 & \textbf{1.001} \\
100 & 9 & 8.900 & 1.369 & \textbf{1.247} \\
100 & 11 & 12.840 & 1.652 & \textbf{1.494} \\\hline
\end{tabular}
\renewcommand{\arraystretch}{1}
\caption{Running times comparison (in seconds) for
Example~\ref{E:BSplineExample}. The source code in C which was used to perform
the tests is available at
\url{https://www.ii.uni.wroc.pl/~pwo/programs/BSpline-BF.c}.}\label{T:TableBS}
\vspace{-3ex}
\end{center}
\end{table*}

\begin{table*}[ht!]\small
\begin{center}
\renewcommand{\arraystretch}{1.3}
\begin{tabular}{l|c|c}
operation~type & new method \\ \hline
$+$ & $(m-1)m(2n-1)+N(m+2)m+NM(m+1)d$ \\ \hline
$-$ & $2(m-1)(4n-1)+n+N((m+1)m+3)$ \\ \hline
$*$ & $2(m-1)m(2n-1)+2N(m+2)m+NM(m+1)d$ \\ \hline
$/$ & $2(m-1)(3n-1)+N(m+2)$ \\ \hline
$(\cdot)^{m-1}$ & $n$
\end{tabular}
\renewcommand{\arraystretch}{1}
\caption{The number of floating-point operations
performed by the new method.}\label{T:TableOps}
\vspace{-3ex}
\end{center}
\end{table*}



\begin{example}\label{E:BSplineExample2}
An experiment similar to Example~\ref{E:BSplineExample}, with a wider choice
of parameters, has been performed. The results have been obtained on the same
computer, software, and precision. More precisely, for each $d \in \{1, 2, 3\}$, 
$n \in \{10, 15, 20, 25, 30, 35, 40, 45, 50\}$, 
$M \in \{1, 2, 3, 4, 5, 10, 15, 20, 25, 30, 50, 100\}$ and $m = 3,4,\ldots,15$, 
a sequence of knots and control points has been generated $100$ times. The 
control points  $\p{W}_{ki} \in [-1, 1]^d$ 
($i=-m,-m+1,\ldots,n-1$, $k=0,1,\ldots,M-1$) and the knot span lengths 
$t_{j+1} - t_j \in [1/50, 1]$ ($j=0,1,\ldots,n-1;\; t_0 = 0$) have been generated
using the \texttt{rand()} C function. The boundary knots are coincident. Each
algorithm is then tested using the same knots and control points. Each curve is
evaluated at $50 \cdot n + 1$ points which are 
$t_j + \ell/50 \times (t_{j+1} - t_j)$ for $j=0,1,\ldots,n-1$ and
$\ell=0,1,\ldots,49$, with the remaining point being $t_n$. Due to the size of 
the table, the resulting running times are available at
\url{https://www.ii.uni.wroc.pl/~pwo/programs/BSpline-BF-Example-5-4.xlsx}.

The results show that the new method is significantly faster than the de Boor-Cox
algorithm except for the case $M = 1$. While the acceleration with respect to the
approach which utilizes Eq.~\eqref{E:BSplineRec} is smaller, it is also
consistent, getting lower running time in all the test cases.

Some statistics regarding the experiments are given in Table~\ref{T:TableBS2}.
\end{example}

\begin{table*}[ht!]\small
\begin{center}
\renewcommand{\arraystretch}{1.45}
\begin{tabular}{lcc}
Algorithm & Total running time [s] & Relative to new method\\\hline
de Boor-Cox & 13993.58 & 6.80\\
eval splines & 2482.96 & 1.21\\
new method & 2058.95 & ---
\end{tabular}\\[2ex]
\begin{tabular}{lccc}
Algorithm & New method win \% & 
\begin{minipage}[t]{2.25cm}
Max time rel.\\
to new method
\end{minipage}\rule[-5mm]{0mm}{4mm}&
\begin{minipage}[t]{2.25cm}
Min time rel.\\
to new method
\end{minipage}\\ \hline
de Boor-Cox & 97.01\% & 11.686 & 0.664 \\
eval splines & 100.00\% & 1.877 & 1.056
\end{tabular}
\renewcommand{\arraystretch}{1}
\vspace{2ex}
\caption{Statistics for Example~\ref{E:BSplineExample2}.
The source code in C which was used to perform the tests is available at
\url{https://www.ii.uni.wroc.pl/~pwo/programs/BSpline-BF.c}.}\label{T:TableBS2}
\end{center}
\end{table*}

\subsection{Evaluating a tensor product B-spline surface}\label{SS:SurfaceFastComp}
Let $n_1, n_2, m_1, m_2 \in \mathbb{N}$. Let
$$
T := (t_{-m_1}, t_{-m_1+1}, \ldots, t_{n_1+m_1}),\quad
V := (v_{-m_2}, v_{-m_2+1}, \ldots, v_{n_2+m_2})
$$
be the knot sequences.
For $i = -m_1, -m_1+1, \ldots, n_1-1$ and $\ell = -m_2, -m_2+1, \ldots, n_2-1$,
let $N_{m_1,i}(u; T)$ be a B-spline function of degree $m_1$ with the knot sequence $T$,
and $N_{m_2, \ell}(w; V)$ be a B-spline function of degree $m_2$ with the knot sequence $V$.
The \textit{tensor product B-spline surface} $\p{S}$
with control points $\p{W}_{i \ell} \in \mathbb{E}^d$
is given by the following formula:
\begin{equation}\label{E:BSplineSurf}
 \p{S}(u,w) := \sum_{i=-m_1}^{n_1-1} \sum_{\ell=-m_2}^{n_2-1} \p{W}_{i \ell} N_{m_1, i}(u; T) N_{m_2, \ell}(w; V),
\end{equation}
with
$(u, w) \in D \equiv [t_0, t_{n_1}] \times [v_0, v_{n_2}]$.

Let $(u,w) \in [t_{j_1}, t_{j_1+1}) \times [v_{j_2}, v_{j_2 + 1})$
for $0 \leq j_1 < n_1$ and $0 \leq j_2 < n_2$.
Then, Eq.~\eqref{E:BSplineSurf} simplifies to
\begin{equation}\label{E:BSplineSurfSimp}
 \p{S}(u,w) = \sum_{i=j_1-m_1}^{j_1} \sum_{\ell=j_2-m_2}^{j_2-1} \p{W}_{i \ell} N_{m_1, i}(u; T) N_{m_2, \ell}(w; V).
\end{equation}

If one needs to evaluate the B-spline surface $\p{S}$ at $N_1 \cdot N_2$ points
$$\{(u_i, w_j): 0 \leq i < N_1 \wedge 0 \leq j < N_2\} \subset D,$$
a similar approach to the one shown for B-spline curves can be used, i.e.,
\begin{enumerate}
 \item For all non-empty knot spans, find the Bernstein-B\'{e}zier coefficients
       of the B-spline basis functions
       in both dimensions --- $O(n_1 m_1^2 + n_2 m_2^2)$ operations.
 \item For each $u_i$ $(0 \leq i < N_1)$, evaluate $N_{m_1, \ell}(u_i; T)$
       for such $\ell$ that the corresponding B-spline basis functions
       do not vanish. For each $w_i$ $(0 \leq i < N_2)$, evaluate $N_{m_2, \ell}(w_i; V)$
       for such $\ell$ that the corresponding B-spline basis functions
       do not vanish. The evaluation can be done using the algorithm given in~\cite{WCh2020}.
       In total, this requires $O(N_1 m_1^2 + N_2 m_2^2)$ operations.
 \item For each $(u,w) \in \{(u_i, w_j): 0 \leq i < N_1 \wedge 0 \leq j < N_2\}$,
       evaluate $\p{S}(u,w)$ --- $O(N_1 N_2 m_1 m_2 d)$ operations.
\end{enumerate}
In total, this procedure requires
$O((N_1 + n_1) m_1^2 + (N_2 + n_2) m_2^2 + N_1 N_2 m_1 m_2 d)$ operations.

\begin{example}\label{Ex:BSplineSurf}
Table~\ref{T:TableBSSurf} shows the comparison between the running times of
the evaluation of points ona a tensor product B-spline surface using the de
Boor-Cox algorithm and the new method based of Algorithm~\ref{A:BSplineBezForm}.

The results have been obtained on a computer with
\texttt{Intel Core i5-6300U CPU} at \texttt{2.40GHz} processor and \texttt{4GB}
\texttt{RAM}, using \texttt{GNU C Compiler 11.2.0} (single precision).

The following numerical experiments have been conducted. For $d = 3$
and $n_1 \in \{10, 30, 50\},$ $n_2 = n_1$, $m_1 \in \{3,5,7,9\}$,
$m_2 \in \{m_1 - 2, m_1\}$,
sequences of knots $T,V$ and control points have been generated $10$ times.
The control points $\p{W}_{i \ell} \in [-1, 1]^3$
$(i=-m_1,-m_1+1,\ldots,n_1-1,\; \ell=-m_2, -m_2+1, \ldots, n_2-1$) and the knot span lengths
$t_{j_1+1} - t_{j_1} \in [1/50, 1]$ ($j_1=0,1,\ldots,n_1-1;\; t_0 = 0$),
$v_{j_2+1} - v_{j_2} \in [1/50, 1]$ ($j_2=0,1,\ldots,n_2-1;\; v_0 = 0$)
have been generated
using the \texttt{rand()} C function. The boundary knots are coincident. Each
algorithm is then tested using the same knots and control points.
Each non-empty knot span in $T$ and $V$ has been sampled $50$ times,
with uniform distances between samples.
Table~\ref{T:TableBSSurf} shows the total running time of all
$10 \cdot (50 n_1 + 1) \cdot (50 n_2 + 1)$ surface evaluations for each method.

Numerical tests show that the new method has,
on average, between $7.19$ (for $n_1 = n_2 = 50$)
and $7.27$ (for $n_1 = n_2 = 10$) common digits
with the result of the
numerically stable
de Boor-Cox algorithm in single precision
($8$ digits) computations.

Table~\ref{T:TableBS2Surf} shows some statistics for the experiment.
\end{example}

\begin{table*}[ht!]\small
\begin{center}
\renewcommand{\arraystretch}{1.3}
\begin{tabular}{ccccc}
$n_1 = n_2$ & $m_1$ & $m_2$ & de~Boor-Cox & new method \\ \hline
10 & 3 & 1 & 1.185 & 0.477 \\
10 & 3 & 3 & 2.610 & 0.879 \\
10 & 5 & 3 & 4.224 & 1.344 \\
10 & 5 & 5 & 7.766 & 1.958 \\
10 & 7 & 5 & 10.796 & 2.592 \\
10 & 7 & 7 & 17.385 & 3.657 \\
10 & 9 & 7 & 22.301 & 4.568 \\
10 & 9 & 9 & 32.807 & 5.362 \\ \hline
30 & 3 & 1 & 10.528 & 4.288 \\
30 & 3 & 3 & 23.451 & 8.010 \\
30 & 5 & 3 & 38.009 & 12.190 \\
30 & 5 & 5 & 69.644 & 17.699 \\
30 & 7 & 5 & 96.775 & 23.408 \\
30 & 7 & 7 & 155.907 & 32.924 \\
30 & 9 & 7 & 200.379 & 41.209 \\
30 & 9 & 9 & 294.417 & 48.231 \\ \hline
50 & 3 & 1 & 29.060 & 11.898 \\
50 & 3 & 3 & 65.101 & 22.325 \\
50 & 5 & 3 & 105.349 & 33.948 \\
50 & 5 & 5 & 193.328 & 49.205 \\
50 & 7 & 5 & 268.871 & 65.005 \\
50 & 7 & 7 & 432.882 & 91.607 \\
50 & 9 & 7 & 554.807 & 114.142 \\
50 & 9 & 9 & 817.361 & 134.059 \\ \hline
\end{tabular}
\renewcommand{\arraystretch}{1}
\caption{Running times comparison (in seconds) for
Example~\ref{Ex:BSplineSurf}. The source code in C which was used to perform
the tests is available at
\url{https://www.ii.uni.wroc.pl/~pwo/programs/BSpline-BF-Surf.c}.}\label{T:TableBSSurf}
\vspace{-3ex}
\end{center}
\end{table*}

\begin{table*}[ht!]\small
\begin{center}
\renewcommand{\arraystretch}{1.45}
\begin{tabular}{lcccc}
Algorithm & Total running time [s] & Ratio &
Max ratio & Min ratio\\ \hline
de Boor-Cox & 3454.94 & 4.73 & 6.12 & 2.44\\
new method & 730.98 & --- & --- & ---
\end{tabular}\\[2ex]
\renewcommand{\arraystretch}{1}
\vspace{2ex}
\caption{Statistics for Example~\ref{Ex:BSplineSurf}.
The source code in C which was used to perform the tests is available at
\url{https://www.ii.uni.wroc.pl/~pwo/programs/BSpline-BF-Surf.c}.}\label{T:TableBS2Surf}
\end{center}
\end{table*}

Additionally, note that if the Bernstein-B\'{e}zier coefficients of the B-spline basis functions
$N_{m_1, i}(u; T), N_{m_2, \ell}(w; V)$ are known
(in the intervals $[t_{j_1}, t_{j_1+1})$ and $[v_{j_2}, v_{j_2 + 1})$, respectiely), i.e.,
$$
N_{m_1, i}(u; T) = \sum_{k=0}^{m_1} b_k^{(i,j_1)} B^{m_1}_k(t) \qquad (u \in [t_{j_1}, t_{j_1 + 1}))
$$
and
$$
N_{m_2, \ell}(w; V) = \sum_{k=0}^{m_2} d_k^{(\ell,j_2)} B^{m_2}_k(v) \qquad (w \in [v_{j_2}, v_{j_2 + 1})),
$$
where
$$
t := \dfrac{u - t_{j_1}}{t_{j_1 + 1} - t_{j_1}}, \qquad
v := \dfrac{w - v_{j_2}}{v_{j_2 + 1} - v_{j_2}},
$$
one can convert the tensor product B-spline patch to a tensor product B\'{e}zier surface:
\begin{equation}
 \p{S}(u,w) =
 \sum_{k_1=0}^{m_1} \sum_{k_2=0}^{m_2} \p{V}_{k_1, k_2}^{(j_1, j_2)}
 B^{m_1}_{k_1}(t) B^{m_2}_{k_2}(v).
\end{equation}
The points
$$
\p{V}_{k_1, k_2}^{(j_1, j_2)} := \sum_{i=j_1-m_1}^{j_1} \sum_{\ell=j_2-m_2}^{j_2-1} b_{k_1}^{(i,j_1)} d_{k_2}^{(\ell,j_2)} \p{W}_{i \ell}
$$
are the control points of a B\'{e}zier patch
for the domain $[t_{j_1}, t_{j_1 + 1}) \times [v_{j_2}, v_{j_2 + 1})$
(cf.~Theorem~\ref{T:BB-Partition1}).

%

\section{Generalizations}                                 \label{S:BSplineGen}

The approach presented in Section~\ref{S:BSplineBezier}
can be generalized so that the
inner knots may have their multiplicity higher than $1$ or the boundary knots 
are of multiplicity lower than $m+1$.

\subsection{Inner knots of any multiplicity}

When an inner knot has multiplicity over $1$, some knot spans $[t_j, t_{j+1})$
($j=0,1,\ldots,n-1$) are empty. It is only necessary to find the B-spline
functions' coefficients over the non-empty knot spans. If there are $n_e$ such
knot spans, one only needs to find $n_e (m+1)^2$ coefficients, and the algorithm
will have $O(n_e m^2)$ complexity. To use the continuity condition, the 
following definition will be useful.
\begin{definition}\label{D:LeftRightNeighbour}
The {\rm{left neighbor}} of a given knot $t_k$ is the knot $t_{\ell}$ if $\ell$
is the largest natural number such that $t_{\ell} < t_k$, i.e., 
$[t_\ell, t_{\ell + 1})$ is non-empty and $t_{\ell+1} = t_k$.

The {\rm{right neighbor}} of a given knot $t_k$ is the knot $t_r$ if $r$ is the
smallest natural number such that $t_k < t_r$, i.e., $[t_{r-1}, t_r)$ is
non-empty and $t_k = t_{r-1}$.
\end{definition}
Note that in the case considered in Section~\ref{S:BSplineBezier}, the right
neighbor of $t_j$ ($j=0,1,\ldots,n-1$) is always $t_{j+1}$.

From Remark~\ref{R:BSplineAssumptions}, it follows that each B-spline function 
is continuous in $(t_0, t_n)$. The only modification then is in the continuity
condition in Eq.~\eqref{E:BSplineStage3}. Let us consider a non-empty knot span
$[t_j, t_{j+1})$ ($j=0,1,\ldots,n-2$). Let $t_r$ be the right neighbor of
$t_{j+1}$, i.e., $t_{r-1} = t_{j+1}$. In this case, the continuity property at
$t_{j+1}$ is
$$
 \sum_{k=0}^m b_k^{(i,j)} B^m_k\Big(\dfrac{t_{j+1} - t_j}
                                          {t_{j+1} - t_j}\Big) =
 \sum_{k=0}^m b_k^{(i,r-1)} B^m_k\Big(\dfrac{t_{j+1} - t_{r-1}}
                                            {t_r - t_{r-1}}\Big),
$$
which simplifies to $b_m^{(i,j)} = b_0^{(i,r-1)}$ (cf.~Eq.~\eqref{E:InitVal}). 
In such case, the recurrence
relation~\eqref{E:BSplineStage3} takes the form
\begin{equation*}
  \left\{\begin{array}{l}
   b_m^{(i,j)} = b_0^{(i,r-1)},\\
   b_k^{(i,j)} = \dfrac{t_j - t_i}{t_{j+1} - t_i} b_{k+1}^{(i,j)}
               + \dfrac{v_i}{t_{j+1} - t_i} \Big(
                   (t_{j+1} - t_{m+i+2}) b_k^{(i+1,j)}
                 + (t_{m+i+2} - t_j) b_{k+1}^{(i+1,j)} \Big)\\
  \qquad \qquad \qquad \qquad \qquad \qquad \qquad \qquad \qquad
   \qquad \qquad \qquad (k=m-1,m-2,\ldots,0)
 \end{array}\right.
\end{equation*}
(cf.~Eq.~\eqref{E:BSplineV}), where $t_r$ is the right neighbor of $t_{j+1}$,
and $j=n-2,n-3,\ldots,0$, $i=j-1,j-2,\ldots,j-m+1$.
It is thus enough to substitute line 15 of Algorithm~\ref{A:BSplineBezForm}
with
$$
B[j,i,m] \gets B[r-1,i,0]
$$
and to skip the iterations of loops over $j$ in lines 3 and 12
if $t_j = t_{j+1}$.
Example~\ref{Ex:BSplineBezMult} presents this approach.
\begin{example}\label{Ex:BSplineBezMult}
Let us set $m := 3$, $n := 5$. Let the knots be
$$
\begin{array}{c|c|c|c|c|c|c|c|c|c|c|c}
    t_{-3} & t_{-2} & t_{-1} & t_0 & t_1 & t_2 & t_3 & t_4 & t_5 &
    t_6 & t_7 & t_8\\
    \hline
    0 & 0 & 0 & 0 & 3 & 3 & 5 & 9 & 10 & 10 & 10 & 10
\end{array}.
$$
The knot $t_1$ is of multiplicity $2$. To compute the adjusted
Bernstein-B\'{e}zier coefficients of the B-spline functions over $[t_0, t_1)$ 
a continuity condition with the knot span $[t_2, t_3)$ is used, as $t_1 = t_2$.
Figure~\ref{F:BSplineBezSchemeMult} illustrates this approach to computing all
necessary coefficients, analogous to Example~\ref{Ex:BSplineBez}.
\end{example}

\begin{figure*}[ht!]
\centering
\begin{tikzpicture}
\foreach \col in {0,...,7}
 {
 \fill[black!20!white] (0, \col) -- (\col, \col) -- (\col, \col + 1) --
                       (0, \col + 1) -- (0, \col);
 \draw[pattern=north east lines] (\col, \col) -- (\col + 1, \col) --
                      (\col + 1, \col + 1) -- (\col, \col + 1) -- (\col, \col);
 \fill[black!20!white] (\col + 4, \col) -- (11, \col) -- (11, \col + 1) --
                       (\col + 4, \col + 1) -- (\col + 4, \col);
 \draw[pattern=horizontal lines] (\col+3, \col) -- (\col+4, \col) --
                      (\col+4, \col+1) -- (\col+3, \col+1) -- (\col+3, \col);
 }
 \fill[black!30!white] (0, 0) -- (3, 0) -- (3, 8) -- (0, 8) -- (0, 0);
 \fill[black!30!white] (8, 0) -- (11, 0) -- (11, 8) -- (8, 8) -- (8, 0);
 \fill[black!30!white] (4,0) -- (5,0) -- (5,8) -- (4,8) -- (4,0);

 \foreach \row in {0,...,11}
 \draw (\row, 0) -- (\row, 8);
\foreach \col in {0,...,8}
 \draw (0, \col) -- (11, \col);

 \node[left] at (0, 0.5) {$N_{3, -3}(u)$};
 \node[left] at (0, 1.5) {$N_{3, -2}(u)$};
 \node[left] at (0, 2.5) {$N_{3, -1}(u)$};
 \node[left] at (0, 3.5) {$N_{3, 0}(u)$};
 \node[left] at (0, 4.5) {$N_{3, 1}(u)$};
 \node[left] at (0, 5.5) {$N_{3, 2}(u)$};
 \node[left] at (0, 6.5) {$N_{3, 3}(u)$};
 \node[left] at (0, 7.5) {$N_{3, 4}(u)$};

 \node[below] at (0,0) {$0$};
 \node[below] at (1,0) {$0$};
 \node[below] at (2,0) {$0$};
 \node[below] at (3, 0) {$0$};
 \node[below] at (4,0) {$3$};
 \node[below] at (5,0) {$3$};
 \node[below] at (6,0) {$5$};
 \node[below] at (7,0) {$9$};
 \node[below] at (8,0) {$10$};
 \node[below] at (9,0) {$10$};
 \node[below] at (10,0) {$10$};
 \node[below] at (11, 0) {$10$};

 \foreach \row in {-3,-2, -1, 1, 2, 3, 4, 6, 7, 8}
  {
    \node[above] at (\row + 3, 8) {$t_{\row}$};
  }

 \draw[very thick] (3, 0) -- (8, 0) -- (8, 8) -- (3, 8) -- (3, 0);
 \node[above] at (3, 8) {$t_{0}$};
 \node[above] at (8, 8) {$t_{5}$};

 \foreach \row in {3.5, 5.5, 6.5}
 {
   \draw[->] (\row, \row - 0.1) -- (\row, \row - 0.9);
 }
 \foreach \row in {3.5, 5.5, 6.5, 7.5}
 {
   \draw[->] (\row, \row - 1.1) -- (\row, \row - 1.9);
 }
 \draw[->] (5.4, 2.5) -- (3.6, 2.5);
 \draw[->] (5.4, 1.5) -- (3.6, 1.5);
 \foreach \row in {5.5, 6.5}
 {
   \draw[->] (\row + 0.9, \row - 1) -- (\row + 0.1, \row - 1);
   \draw[->] (\row + 0.9, \row - 2) -- (\row + 0.1, \row - 2);
 }
 \draw[->] (9.2, 5.5) -- (7.6, 5.5);
 \draw[->] (9.2, 6.5) -- (7.6, 6.5);
  \draw[->] (7.5, 7.4) -- (7.5, 6.6);

 \fill[white] (8.3, 5.25) -- (10.7, 5.25) -- (10.7, 5.8) -- (8.3, 5.8) -- 
                                                            (8.3, 5.25); 
 \draw (8.3, 5.25) -- (10.7, 5.25) -- (10.7, 5.8) -- (8.3, 5.8) -- (8.3, 5.25); 
 \node[right] at (8.4, 5.5) {$N_{3,2}(10) = 0$};
 \fill[white] (8.3, 6.25) -- (10.7, 6.25) -- (10.7, 6.8) -- (8.3, 6.8) -- 
                                                            (8.3, 6.25); 
 \draw (8.3, 6.25) -- (10.7, 6.25) -- (10.7, 6.8) -- (8.3, 6.8) -- (8.3, 6.25); 
 \node[right] at (8.4, 6.5) {$N_{3,3}(10) = 0$};
\end{tikzpicture}
\caption{An illustration of Example~\ref{Ex:BSplineBezMult}.}
\label{F:BSplineBezSchemeMult}
\end{figure*}

\subsection{Boundary knots of multiplicity lower than 
                                                  \texorpdfstring{$m+1$}{m+1}}

First, note that in Section~\ref{S:BSplineBezier}, only the assumption that
$t_n = t_{n+m}$ is used, therefore if that condition holds,
Theorem~\ref{T:BSplineBezierBig} and Algorithm~\ref{A:BSplineBezForm} still
apply, regardless of the multiplicity of boundary knots
$t_{-m}, t_{-m+1}, \ldots, t_0$.

If the boundary knot $t_n$ has multiplicity lower than $m+1$ (i.e., the knot sequence is unclamped),
the problem can
be reduced so that it can be solved using Theorem~\ref{T:BSplineBezierBig}.
Its drawback, however, is higher complexity.

The idea is to \textit{inflate} the multiplicity of $t_{n+m}$ up to $m+1$. More
precisely, let $t_{n+m-\ell-1} < t_{n+m-\ell} = t_{n+m}$, i.e., $t_{n+m-\ell}$ 
has multiplicity $\ell + 1$. Let the $m-\ell$ new knots
$t_{n+m+1} = t_{n+m+2} = \ldots = t_{n+2m-\ell}$ be defined so that 
$t_{n+m} = t_{n+m+1}$. This allows to execute Algorithm~\ref{A:BSplineBezForm}
with the new arguments $n_1 := n+m-\ell$, $m_1 := m$ and
the \textit{inflated} knot sequence
$$
 \underbrace{t_{-m} \leq \ldots \leq t_{-1} \leq t_0}_\text{boundary knots}
 \leq
 \underbrace{t_1 < \ldots < t_{n+m-\ell-1}}_\text{inner knots}
 <
 \underbrace{t_{n+m-\ell} = t_{n+m-\ell+1} = \ldots =
             t_{n+2m-\ell}}_\text{boundary knots}.
$$
It remains then to return the coefficients of $N_{mi}$ over $[t_j, t_{j+1})$
for $j=0,1,\ldots,n-1$ and $i=j-m,j-m+1,\ldots,j$. This approach requires the
computation of $O((n+m-\ell)m^2)$ coefficients and is presented in
Example~\ref{Ex:BSplineBezExt}.
\begin{example}\label{Ex:BSplineBezExt}
Let us set $m := 3$, $n := 2$. Let the knots be
$$
\begin{array}{c|c|c|c|c|c|c|c|c}
    t_{-3} & t_{-2} & t_{-1} & t_0 & t_1 & t_2 & t_3 & t_4 & t_5 \\
    \hline
    -3 & -2 & -1 & 0 & 1 & 2 & 3 & 4 & 5
\end{array}.
$$
After adding the knots $t_6=t_7=t_8$ such that $t_5 = t_8$ (thus increasing $n$
by $3$), the problem takes the form
$$
\begin{array}{c|c|c|c|c|c|c|c|c|c|c|c}
    t_{-3} & t_{-2} & t_{-1} & t_0 & t_1 & t_2 & t_3 & t_4 & t_5 &
    t_6 & t_7 & t_8 \\
    \hline
    -3 & -2 & -1 & 0 & 1 & 2 & 3 & 4 & 5 & 5 & 5 & 5
\end{array}.
$$
Figure~\ref{F:BSplineBezSchemeExt} illustrates the application of
Algorithm~\ref{A:BSplineBezForm} (cf.~Example~\ref{Ex:BSplineBez}) computing all
the adjusted Bernstein-B\'{e}zier coefficients of the inflated problem. The
coefficients which are relevant to the solution of the primary problem are in 
the frame drawn in bold.
\end{example}

\begin{figure*}[ht!]
\centering
\begin{tikzpicture}
\foreach \col in {0,...,7}
 {
 \fill[black!20!white] (0, \col) -- (\col, \col) -- (\col, \col + 1) --
                       (0, \col + 1) -- (0, \col);
 \draw[pattern=north east lines] (\col, \col) -- (\col + 1, \col) --
                      (\col + 1, \col + 1) -- (\col, \col + 1) -- (\col, \col);
 \fill[black!20!white] (\col + 4, \col) -- (11, \col) -- (11, \col + 1) --
                       (\col + 4, \col + 1) -- (\col + 4, \col);
 \draw[pattern=horizontal lines] (\col+3, \col) -- (\col+4, \col) --
                      (\col+4, \col+1) -- (\col+3, \col+1) -- (\col+3, \col);
 }
 \fill[black!30!white] (0, 0) -- (3, 0) -- (3, 8) -- (0, 8) -- (0, 0);
 \fill[black!30!white] (8, 0) -- (11, 0) -- (11, 8) -- (8, 8) -- (8, 0);

 \foreach \row in {0,...,11}
 \draw (\row, 0) -- (\row, 8);
\foreach \col in {0,...,8}
 \draw (0, \col) -- (11, \col);

 \node[left] at (0, 0.5) {$N_{3, -3}(u)$};
 \node[left] at (0, 1.5) {$N_{3, -2}(u)$};
 \node[left] at (0, 2.5) {$N_{3, -1}(u)$};
 \node[left] at (0, 3.5) {$N_{3, 0}(u)$};
 \node[left] at (0, 4.5) {$N_{3, 1}(u)$};
 \node[left] at (0, 5.5) {$N_{3, 2}(u)$};
 \node[left] at (0, 6.5) {$N_{3, 3}(u)$};
 \node[left] at (0, 7.5) {$N_{3, 4}(u)$};

 \node[below] at (0,0) {$-3$};
 \node[below] at (1,0) {$-2$};
 \node[below] at (2,0) {$-1$};
 \node[below] at (3, 0) {$0$};
 \node[below] at (4,0) {$1$};
 \node[below] at (5,0) {$2$};
 \node[below] at (6,0) {$3$};
 \node[below] at (7,0) {$4$};
 \node[below] at (8,0) {$5$};
 \node[below] at (9,0) {$5$};
 \node[below] at (10,0) {$5$};
 \node[below] at (11, 0) {$5$};

 \foreach \row in {-3,-2, -1, 1, 3, 4, 5, 6, 7, 8}
  {
    \node[above] at (\row + 3, 8) {$t_{\row}$};
  }

 \draw[very thick] (3, 0) -- (5, 0) -- (5, 5) -- (3, 5) -- (3, 0);
 \node[above] at (3, 8) {$t_{0}$};
 \node[above] at (5, 8) {$t_{2}$};

 \foreach \row in {3.5, 4.5, 5.5, 6.5}
 {
   \draw[->] (\row, \row - 0.1) -- (\row, \row - 0.9);
 }
 \foreach \row in {3.5, 4.5, 5.5, 6.5, 7.5}
 {
   \draw[->] (\row, \row - 1.1) -- (\row, \row - 1.9);
 }
 \draw[->] (4.4, 2.5) -- (3.6, 2.5);
 \draw[->] (4.4, 1.5) -- (3.6, 1.5);
 \foreach \row in {3.5, 4.5, 5.5, 6.5}
 {
   \draw[->] (\row + 0.9, \row - 1) -- (\row + 0.1, \row - 1);
   \draw[->] (\row + 0.9, \row - 2) -- (\row + 0.1, \row - 2);
 }
 \draw[->] (9.2, 5.5) -- (7.6, 5.5);
 \draw[->] (9.2, 6.5) -- (7.6, 6.5);
  \draw[->] (7.5, 7.4) -- (7.5, 6.6);

 \fill[white] (8.3, 5.25) -- (10.7, 5.25) -- (10.7, 5.8) -- (8.3, 5.8) -- 
                                                            (8.3, 5.25); 
 \draw (8.3, 5.25) -- (10.7, 5.25) -- (10.7, 5.8) -- (8.3, 5.8) -- (8.3, 5.25); 
 \node[right] at (8.4, 5.5) {$N_{3,2}(5) = 0$};
 \fill[white] (8.3, 6.25) -- (10.7, 6.25) -- (10.7, 6.8) -- (8.3, 6.8) -- 
                                                            (8.3, 6.25); 
 \draw (8.3, 6.25) -- (10.7, 6.25) -- (10.7, 6.8) -- (8.3, 6.8) -- (8.3, 6.25); 
 \node[right] at (8.4, 6.5) {$N_{3,3}(5) = 0$};
\end{tikzpicture}
\caption{An illustration of Example~\ref{Ex:BSplineBezExt}.}
\label{F:BSplineBezSchemeExt}
\end{figure*}
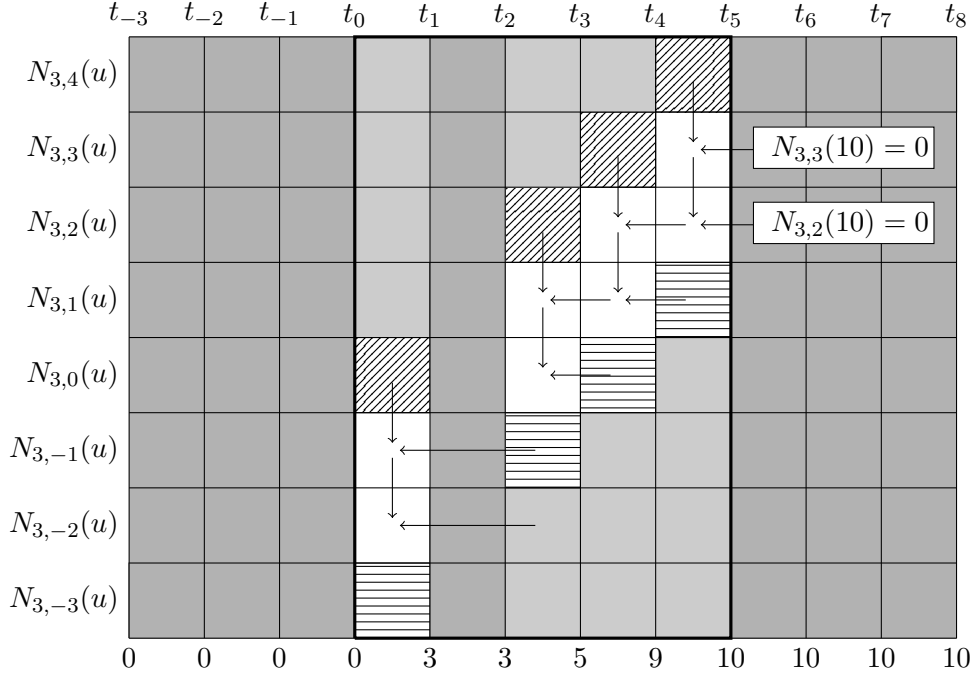

\section{Special case: uniform knots}\label{S:UniformKnots}
In the case of uniform knots, i.e.,
$$t_i := t_0 + i \cdot h \qquad (h > 0,\, i=-1, -m+1,\ldots,n+n),$$
one can check that
$$
N_{mi}([t_j, t_{j+1})) \equiv N_{m,i+1}([t_{j+1}, t_{j+2})).
$$
This means that, to solve Problem~\ref{P:BSpline1}, it is enough to find
the coefficients of B-spline basis functions over one knot span.

Let us set $j \in \mathbb{N}$ such that $0 \leq j \leq n - 1$.
Equation~\eqref{E:BSplineStage1Exp} takes the form
\begin{equation}\label{E:BSplineUniLast}
\left\{\begin{array}{ll}
  b_k^{(j,j)} = 0 & (k = 0,1,\ldots,m-1),\\
  b_m^{(j,j)} = \dfrac{1}{m!},&
\end{array}\right.
\end{equation}
while Eq.~\eqref{E:BSplineStage1LastExp} simplifies to
\begin{equation}\label{E:BSplineUniFirst}
\left\{\begin{array}{ll}
  b_0^{(j-m,j)} = \dfrac{1}{m!},&\\
  b_k^{(j-m,j)} = 0 & (k = 1,2,\ldots,m).
\end{array}\right.
\end{equation}

In order to find the coefficients of the basis functions
$N_{mi}$ for $i=j-m+1,j-m+2,\ldots,j-1$, one can use Theorem~\ref{T:BSplineBezCase1},
which gives $m$ equations of the form
\begin{equation}\label{E:BSplineBezCase1ThUni}
(j+1 - i) b^{(i,j)}_k + (i - j) b^{(i,j)}_{k+1} =
(j-m-i-1) b^{(i+1,j)}_k + (m+i+2-j) b^{(i+1,j)}_{k+1}
                                                    \qquad (k=0,1,\ldots,m-1).
\end{equation}
To complete the recurrence scheme, it is enough to use the continuity condition
$$
N_{m,i+1}(t_{j+1}^-) = N_{m,i+1}(t_{j+1}^+),
$$
which, due to the knot uniformity, can be formulated as
$$
N_{m,i+1}(t_{j+1}^-) = N_{mi}(t_{j}^+),
$$
which, eventually, gives
$$
b_0^{(i+1,j)} = b_m^{(i,j)}.
$$
The complete recurrence scheme is thus,
for $i = j - 1, j - 2, \ldots, j - m + 1$,
\begin{equation}\label{E:BSplineUniScheme}
\left\{\begin{array}{l}
  b_m^{(i,j)} = b_0^{(i+1,j)},\\
  b^{(i,j)}_k  = \dfrac{j - i}{j - i + 1} b^{(i,j)}_{k+1}
               +  \dfrac{j - i - m - 1}{j - i + 1} b^{(i+1,j)}_k
               + \dfrac{m + i + 2 - j}{j - i + 1} b^{(i+1,j)}_{k+1} \\
\qquad \qquad \qquad \qquad \qquad \qquad \qquad \qquad \qquad \qquad \qquad \qquad (k=m-1, m-2,\ldots,0).
\end{array}\right.
\end{equation}
Note that both the coefficients of B-spline basis functions
given in Eqs.~\eqref{E:BSplineUniLast} and~\eqref{E:BSplineUniFirst},
as well as the coefficients in
the recurrence scheme~\eqref{E:BSplineUniScheme} are rational.
This means that all the coefficients of the B-spline basis functions
are thus rational and can be computed without errors.
This approach requires $O(m^2)$ operations, which is optimal.

\section{Conclusion}
We have discovered a new differential-recurrence relation satisfied by B-spline functions
of the same degree.
The relation is a foundation of a new asymptotically optimal method of finding
the Bernstein-B\'{e}zier coefficents of all B-spline basis functions in the clamped case
over all knot spans.
The algorithm can be generalized for different knot multiplicities,
including the unclamped case.
The new method allows to accelerate the computations of B-spline basis functions,
which leads to faster evaluation of B-spline curves and surfaces.
Numerical experiments show that the algorithm is stable.
Further research is required to find the application of
the new differential-recurrence relation in finding
the Bernstein-B\'{e}zier coefficients of
B-spline basis functions over a single knot span.

\bibliographystyle{elsart-num-sort}
\bibliography{BSplineBF-elsart}

\end{document}